\documentclass[reqno, 12pt]{amsart}
\usepackage{amsmath,amsxtra,amssymb,latexsym, amscd,amsthm}
\usepackage[mathscr]{euscript}
\usepackage{mathrsfs}
\usepackage[english]{babel}
\usepackage{color}
\usepackage[active]{srcltx}
\usepackage{mathtools,stmaryrd}
\usepackage{enumerate}
\usepackage{hyperref}

\usepackage[T1]{fontenc}

\newcommand{\RR}{{\mathbb R}}
\newcommand{\CC}{{\mathbb C}}
\newcommand{\N}{\mathbb{N}}
\newcommand{\PP}{\mathbb{P}}

\newcommand{\cri}{\mathbf{Crit}}
\newcommand{\ve}{\mathrm{vec}}
\newcommand{\res}{\mathrm{Res}}
\newcommand{\mc}{\mathscr{C}}

\newcommand{\diag}{\mathrm{diag}}
\newcommand{\td}[1]{\tilde{#1}}

\newcommand{\rank}{\mbox{\upshape rank}}

\newcommand{\wt}[1]{\widetilde{#1}}

\newcommand{\co}{\text{\upshape co}}

\newtheorem {theorem}{Theorem}[section]

\newtheorem {corollary}{Corollary}[section]
\newtheorem {lemma}{Lemma}[section]
\newtheorem {example}{Example}[section]
\newtheorem {definition}{Definition}[section]
\newtheorem {remark}{Remark}[section]
\newtheorem {proposition}{Proposition}[section]

\def\EES{{\accent"5E e}\kern-.5em\raise.8ex\hbox{\char'23 }}
\def\ow{o\kern-.42em\raise.82ex\hbox{
   \vrule width .12em height .0ex depth .075ex \kern-0.16em \char'56}\kern-.07em}
\def\OW{o\kern-.460em\raise1.36ex\hbox{
\vrule width .13em height .0ex depth .075ex \kern-0.16em
\char'56}\kern-.07em}
\def\DD{D\kern-.7em\raise0.4ex\hbox{\char '55}\kern.33em}
\title{On continuous selections of polynomial functions}

\author{Feng Guo$^*$}
\address[Feng Guo]{School of Mathematical Sciences, Dalian  University of Technology, Dalian, 116024, China}
\email{fguo@dlut.edu.cn}
\thanks{$^*$Corresponding Author.}

\author{Liguo Jiao}
\address[Liguo Jiao]{School of Mathematical Sciences, Soochow University, Suzhou 215006, Jiangsu Province, China}
\email{hanchezi@163.com}

\author{Do Sang Kim} 
\address[Do Sang Kim]{Department of Applied Mathematics, Pukyong National University, Busan, 48513, Korea}
\email{dskim@pknu.ac.kr}
\date{ \today}

\begin{document}

\begin{abstract} 
	A continuous selection of polynomial functions is a continuous
	function whose domain can be partitioned into finitely many pieces
	on which the function coincides with a polynomial. Given a set of
	finitely many polynomials, we show that there are only finitely
	many continuous selections of it and each one is semi-algebraic.
	Then, we establish some generic properties regarding the critical
	points, defined by the Clarke subdifferential, of these continuous
	selections. In particular, given a set of finitely many
	polynomials with generic coefficients, we show that the critical
	points of all continuous selections of it are finite and the
	critical values are all different, and we also derive the
	coercivity of those continuous selections which are bounded from
	below. We point out that some existing results about
	{\L}ojasiewicz's inequality and error bounds for the maximum
	function of some finitely many polynomials are also valid for all
	the continuous selections of them. 
%	A continuous selection of polynomial functions is a continuous
%	function whose domain can be partitioned into finitely
%	many pieces on which the function coincides with a polynomial.  
%	Given a set of finitely many polynomials, we study some properties
%	satisfied by all continuous selections of it. Especially, we
%	show that there are only finitely many of such continuous
%	selections and each of them is semi-algebraic. Then, we
%	establish some generic properties regarding the critical
%	points, defined by the Clarke subdifferential, of these continuous
%	selections. In particular, given a set of finitely many polynomials with generic
%	coefficients, we show that the critical points of all continuous
%	selections of it are finite and the corresponding critical values
%	are different from each other, and we also
%	derive the coercivity of those continuous selections which are
%	bounded from below. At last, we point out that some existing results
%about non-smooth {\L}ojasiewicz's inequality and error bounds for the
%	maximum function of some finitely many polynomials are also valid
%	for all the continuous selections of these polynomials. 
%This paper presents some properties of semi-algebraic programs. These
%generalize the results established recently in the papers
%\cite{Cui2018-1, Cui2018-2}.
\end{abstract}

%\subjclass[2010]{65K05, 68W30}
\maketitle

\section{Introduction}
In this paper, we mainly study the set of all {\itshape continuous
selections} of some given finitely many multivariate polynomials with
real coefficients. A function is in such a set if it is continuous
and its value at every point is equal to one of the values of those
polynomials at that point. We say the set of the given polynomials the
set of {\itshape selection functions}. This kind of functions belongs to a larger
one called $PC^l$ function which is everywhere locally a continuous
selection of $C^l$-functions (i.e., $l$-times continuously differentiable
functions). The set of $PC^l$ functions covers
various types of functions composed by the selection functions, among
which are the typical examples of the maximum and minimum functions
occurring in optimization. Moreover,
superposition, scalar multiples and finite sums of $PC^l$ functions
are again $PC^l$ functions. In particular, $PC^1$ functions are also called
{\itshape piecewise differentiable} functions.

$PC^l$ functions have many
applications to solution methodology in optimization, particularly in
connection with complementarity problems and variational inequalities.
Jongen and Pallaschke \cite{JP1988} introduced the notion of
continuous selections of differentiable functions to extend the
classical critical point theory to nonsmooth functions. 
Continuous selections of locally Lipschitz continuous functions have
been studied in Hager's paper \cite{Hager1979}.  As a continuous
selection of locally Lipschitz continuous functions is again locally
Lipschitz (c.f. \cite[Corollary 4.1.1]{ScholtesBook}), its critical
point can be defined naturally via the Clarke subdifferential \cite{Clarke1990}.
Womersley \cite{Womersley1982} investigated the optimality conditions
on critical points for piecewise differentiable functions. The
connection between piecewise differentiable functions and nonsmooth
optimization problems is extensively studied in \cite{CHANEY1990649}.
The representations of continuous selections of affine functions and
the topological classification of continuous selections of linear
functions are investigated in \cite{BARTELS1995385}. 
Qualitative aspects of the second order approximation scheme for
regular $PC^2$-functions are considered in \cite{KUNTZ1995197}. 
We refer the readers to
\cite{APS1997,PD1996,DS1997,R2003,ScholtesBook} and the recent
\cite{Cui2018-2, Cui2018-1} for more basic background and developments
in this subject.

In this paper, we restrict our attention to the set of continuous
selections of {\it polynomial} functions and we call such a continuous
selection a CSP function for
short. What we benefit from this restriction are the applications of
deep theory and powerful tools from semi-algebraic geometry to derive
many favorable properties enjoyed by CSP functions. In particular, we
show that there are only finitely many CSP functions
selected from a given set of finitely many polynomials and each one is
semi-algebraic. As a semi-algebraic function, the set of isolated
local minimizers of a CSP function coincides with its set of strictly
local minimizers and both are finite. Obviously, it is not the case for
continuous selections of general functions. 

Over the past few decades, generic properties for
mathematical programming problems have been extensively studied in the
literature, see  
\cite{AHO1997,BDL,DIL,FO1982,LeePham2016,LeePham2017,NieFiniteCon,PT2001,SS1973,SA1997,SR1979}.
The first order necessary optimality condition for a CSP function
states that a local minimizer must be a critical point, i.e., $0$
belongs to the Clarke subdifferential of the CSP function at this
point (c.f. \cite[Theorem 3.1]{Womersley1982}). 
Inspired by the papers \cite{LeePham2016,LeePham2017} by Lee
and Ph\d{a}m where the genericity of semi-algebraic programs is
investigated, we next establish some generic properties concerned with
the set of critical points of all CSP functions with the same set of 
selection functions. Here, the term ``genericity'' means that the
properties hold in the following sense. If we fix the number $r$ and a
degree bound $d$ of the polynomials in the set of selection functions, we
can identify the set of selection functions with the vector of all
coefficients of the $r$ ordered polynomials in the canonical monomial
basis of the space of polynomials
of degree up to $d$. Then, there exists an open and dense
semi-algebraic subset of the vector space such that for each set of
selection functions corresponding to a vector in this subset, these
properties hold for all CSP functions selected from it. In particular,
we obtain the following generic properties for all CSP functions
selected from the same set of finitely many polynomials:
(i) the critical points of all those CSP functions are finite and the
corresponding critical values are all distinct; 
(ii) each of those CSP functions is ``good at infinity'' in
the sense that its non-smooth slope, defined by the Clarke
subdifferential, at a point is not smaller than a
positive constant $c$ whenever the Euclidean norm of the point is
larger than a constant $R$; (iii) each of those CSP functions which are
bounded from below is coercive and hence its global minimum is
attainable at a unique minimizer.
%As a consequence, if a CSP function is bounded from below,
%then its global minimum is attainable at a unique minimizer
%{\color{red}{generically.}}

An error bound for a subset of an Euclidean space is an inequality
that bounds the distance from an arbitrary point in a test set to the
subset in terms of the amount of ``constraint violation'' at that point.
Among the numerous applications of error bounds, they can be used to
estimate the rate of convergence of many optimization methods. 
We refer the readers to \cite{Pang1997} for an excellent survey in
this subject and to the more recent papers
\cite{Aze,BNTPS,DHP2017,FHKO,Kruger2015,Li2013,LMNP,LMP2015,NZ2001} with the
bibliographies therein.
In the papers \cite{DHP2017,LMP2015}, a non-smooth {\L}ojasiewicz's
inequality about the non-smooth slope, defined by the limiting
subdifferential, is established for the maximum function of finitely
many polynomials. Then, some local and global H{\"o}lderian error
bounds with explicit exponents for a polynomial system are obtained.
Note that the maximum function of finitely many polynomials is an
instance of CSP functions selected from these polynomials. Moreover,
the non-smooth slope for the maximum function defined via limiting
subdifferential and the non-smooth slope for any CSP function
defined via Clarke subdifferential have the same representation (see
Section \ref{sec::errorbound}). As a
result, we point out that some results obtained in
\cite{DHP2017,LMP2015} about non-smooth
{\L}ojasiewicz's inequality and error bounds for the maximum function
of finitely many polynomials are also valid for any CSP functions
selected from them. 

\vskip 5pt
The paper is organized as follows. In Section \ref{sec::pre}, we
introduce some notation and preliminaries used in the paper. We
present some basic properties satisfied by all CSP functions in
Section \ref{sec::CSS}. Some
generic properties for CSP functions are establised in Section
\ref{sec::genericity}. In Section \ref{sec::errorbound}, we discuss
some results about non-smooth {\L}ojasiewicz's inequality and error
bounds for CSP functions. In Section \ref{sec::con}, some
conclusions are given.

\section{Preliminaries}\label{sec::pre}

We use the following notation and terminology.
The symbol $\RR$ (resp. $\CC$, $\N$) denotes the set of real (resp.
complex, natural) numbers.
We denote by $\RR_{>0}$ the set of positive real numbers.
$\RR[x]=\RR[x_1,\ldots,x_n]$ denotes the ring of polynomials in variables
$x=(x_1,\ldots,x_n)$ with real coefficients. 
The Euclidean space $\mathbb{R}^n$ is equipped with the usual scalar
product $\langle \cdot, \cdot \rangle$ and the corresponding Euclidean
norm $\|\cdot\|.$ 
%For convenience, let $\Vert
%x\Vert^2:=x_1^2+\cdots+x_n^2$ for any $x\in\CC^n$. 
%Denote $\RR^{n\times n}$ (resp., $\CC^{n\times n}$) as the set of $n\times n$ matrices
%with real (resp., complex) number entries.
%Denote by $\Vert A\Vert$ the $2$-norm of a matrix $A\in\RR^{n\times n}$.
%For $R > 0$, 
For any set $J$, we denote by $\# J$ the cardinality of $J$.
The notation $C^p$ means $p$-times continuously differentiable;
$C^\infty$ is infinitely continuously differentiable.
In what follows, $\RR^n$ will always be considered with its Euclidean
topology, unless stated otherwise.
For a subset $S \subset \mathbb{R}^n$, the closure and convex hull of $S$ in $\RR^n$
are denoted by $\bar{S}$ and $\co\, S,$ respectively. 
Denote by $\mathbb{B}$ (resp., $\bar{\mathbb{B}}$) the unit (resp.,
closed) ball centered at the origin in $\RR^n$. For any
$\alpha\in\N^n$, denote $|\alpha|:=\alpha_1+\cdots+\alpha_n$. 
%with center $x$ (resp., $0$) and radius $R.$ 
%; $C^\infty$ is infinitely continuously differentiable.
%If $f, g$ are two functions with suitably chosen domains and codomains, then
%$f\circ g$ denotes the composite function of $f$ and $g$.

\subsection{Semi-algebraic geometry}
Let us recall some notion and results from semi-algebraic geometry (see,
for example, \cite{RASS, Bochnak1998}).

\begin{definition}{\rm
\begin{enumerate}
  \item[(i)] A subset of $\mathbb{R}^n$ is said to be {\em semi-algebraic} if
	  it is a finite union of sets of the form
$$\{x \in \mathbb{R}^n \colon f_i(x) = 0, i = 1, \ldots, k;\ f_i(x) > 0,
i = k + 1, \ldots, p\},$$
where all $f_{i}$'s are in $\RR[x]$.
 \item[(ii)]
Let $A \subset \Bbb{R}^n$ and $B \subset \Bbb{R}^m$ be semi-algebraic
sets. A map $F \colon A \to B$ is said to be {\em semi-algebraic} if
its graph
$$\{(x, y) \in A \times B \colon y = F(x)\}$$
is a semi-algebraic subset in $\Bbb{R}^n\times\Bbb{R}^m.$
\end{enumerate}
}\end{definition}

%In what follows, we recall some notions and results of semi-algebraic geometry, which can be found in \cite[Chapter 1]{HaHV2017}.
%
%\begin{definition}{\rm
%A subset of $\mathbb{R}^n$ is called {\em semi-algebraic} if it is a finite union of sets of the form $\{x\in\mathbb{R}^n:f_i(x)=0, \ i=1,\ldots,k, \ f_i(x)>0, \ i=k+1,\ldots, p\},$ where all $f_i$ are polynomials. If $A\subset \mathbb{R}^n$ and $B\subset\mathbb{R}^p$ are semi-algebraic sets, then the map $F\colon A\to B$ is said to be {\em semi-algebraic} if its graph $\{(x,y)\in A\times B : y=F(x)\}$ is a semi-algebraic subset in $\mathbb{R}^n\times\mathbb{R}^p.$}
%\end{definition}

Note that semi-algebraic sets and functions enjoy a number of
remarkable properties. We summarize some of the important properties
which will be used in the sequel.

\begin{proposition}\label{property of SA}
The following statements hold$:$
\begin{itemize}
\item[{\rm (i)}] Each semi-algebraic set in $\mathbb{R}$ is a finite union of intervals and points.
\item[{\rm (ii)}] Finite union $($resp.$,$ intersection$)$ of semi-algebraic sets is semi-algebraic.
\item[{\rm (iii)}]The Cartesian product $($resp.$,$ complement$,$ closure$,$ interior$)$ of semi-algebraic sets is semi-algebraic.
\item[{\rm (iv)}] If $f,$ $g$ are semi-algebraic functions on $\mathbb{R}^n$ and $\lambda\in\mathbb{R},$ then $f+g,$ $fg$ and $\lambda f$ are all semi-algebraic functions.
\item[{\rm (v)}] If $f$ is a semi-algebraic function on $\mathbb{R}^n$ and $\lambda\in\mathbb{R},$ then $\{x\in \mathbb{R}^n : f(x)\le\lambda\},$ $\{x\in \mathbb{R}^n : f(x)<\lambda\}$ and $\{x\in \mathbb{R}^n : f(x)=\lambda\}$ are all semi-algebraic sets.
\item[{\rm (vi)}] A composition of semi-algebraic maps is a semi-algebraic map.
%\item[{\rm (vii)}] The closure and the interior of a semi-algebraic set are semi-algebraic sets.
\end{itemize}
\end{proposition}

%\begin{theorem}[Tarski--Seidenberg Theorem]\label{Tarski Seidenberg Theorem}
%The image and inverse image of a semi-algebraic set under a semi-algebraic map are semi-algebraic sets.
%In particular$,$ the projection of a semi-algebraic set is still a semi-algebraic set.
%\end{theorem}
\begin{theorem}[Tarski--Seidenberg Theorem] \label{Tarski Seidenberg Theorem}
The image of a semi-algebraic set by a semi-algebraic map is semi-algebraic.
\end{theorem}

\begin{remark}\label{Remark1}{\rm
If $A\subset \mathbb{R}^n,$ $B\subset\mathbb{R}^m,$ and $C
\subset\mathbb{R}^n\times\mathbb{R}^m$ are semi-algebraic sets, then
we see that
$U:=\{x\in A:(x,y)\in C, \ \forall y\in B\}$ is also a semi-algebraic
set.
To see this, from Proposition \ref{property of SA} (iii) and
Theorem~\ref{Tarski Seidenberg Theorem}, we see that
$\{x\in A: \exists y\in B \textrm{ s.t. } (x,y)\not\in C\}$ is
semi-algebraic.
As the complement of $U$ is the union of the complement of $A$ and the
set $\{x\in A :\exists y\in B \textrm{ s.t. } (x,y)\notin C\},$
it follows that the complement of $U$ is semi-algebraic by
Proposition~\ref{property of SA}(iii).
Thus, $U$ is also semi-algebraic.
%by Proposition~\ref{property of SA}(iii).
In general, if we have a finite collection of semi-algebraic sets,
then any set obtained from them by a finite chain of quantifiers is
also semi-algebraic.
}\end{remark}

%The class of semi-algebraic sets is closed under taking finite
%intersections, finite unions, and complements; a Cartesian product of
%semi-algebraic sets is a semi-algebraic set. Moreover, a major fact
%concerning the class of semi-algebraic sets is its stability under
%linear projections (see, for example,~\cite{Benedetti1990,Bochnak1998}).

%By the Tarski--Seidenberg Theorem, it is not hard to see that the
%closure and the interior of a semi-algebraic set are semi-algebraic
%sets.

Recall the Curve Selection Lemma which will be used in this paper
(see \cite{HaHV2017,Milnor1968}).

\begin{lemma}[Curve Selection Lemma]\label{CurveSelectionLemma}
Let $A$ be a semi-algebraic subset of $\mathbb{R}^n,$ and $u^* \in
\overline{A}\setminus A.$ Then there exists a real analytic
semi-algebraic curve
$$\phi \colon (-\epsilon, \epsilon) \to {\mathbb R}^n$$
with $\phi(0) = u^*$ and with $\phi(t) \in A$ for $t \in (0, \epsilon).$
\end{lemma}

In what follows, we will need the following useful results (see, for example, \cite{Dries1996}).

\begin{lemma}[Monotonicity Lemma] \label{MonotonicityLemma}
Let $a < b$ in $\mathbb{R}.$ If $f \colon [a, b] \rightarrow
\mathbb{R}$ is a semi-algebraic function, then there is a partition $a
=: t_1 < \cdots < t_{N} := b$ of $[a, b]$ such that $f|_{(t_l, t_{l +
1})}$ is $C^1,$ and either constant or strictly monotone$,$ for $l \in
\{1, \ldots, N - 1\}.$
\end{lemma}

The next theorem (see \cite{Bochnak1998,Dries1996}) uses the concept
of a cell whose definition we omit. We do not need the specific
structure of cells described in the formal definition. For us, it will
be sufficient to think of a $C^p$-cell of dimension $r$ as of an
$r$-dimensional $C^p$-manifold, which is the image of the cube $(0,
1)^r$ under a semi-algebraic $C^p$-diffeomorphism. As follows from the
definition, an $n$-dimensional cell in $\mathbb{R}^n$ is an open set.

\begin{theorem}[Cell Decomposition Theorem] \label{CellTheorem}
Let $A \subset \mathbb{R}^n$ be a  semi-algebraic set. Then$,$ for any
$p \in \mathbb{N},$ $A$ can be represented as a disjoint union of a
finite number of cells of class~$C^p.$
\end{theorem}

By Cell Decomposition Theorem, for any $p \in \mathbb{N}$ and any
nonempty semi-algebraic subset $A$ of $\mathbb{R}^n,$ we can write $A$
as a disjoint union of finitely many semi-algebraic $C^p$-manifolds of
different dimensions. The {\em dimension} $\dim A$ of a nonempty
semi-algebraic set $A$ can thus be defined as the dimension of the
manifold of highest dimension of its decomposition. This dimension is
well defined and independent of the decomposition of $A.$ By
convention, the dimension of the empty set is taken to be negative
infinity. We will need the following result (see
\cite{Bochnak1998,Dries1996}).

\begin{proposition} \label{DimensionProposition}
\begin{enumerate}
\item [{\rm (i)}] Let $A \subset \mathbb{R}^n$ be a semi-algebraic set and
$f \colon A \to\mathbb{R}^m$ a semi-algebraic map. Then$,$ $\dim f(A)\leq\dim  A.$

\item [{\rm (ii)}] Let $A \subset \mathbb{R}^n$ be a nonempty
	semi-algebraic set. Then$,$ $\dim(\bar{A} \setminus A) < \dim A.$
	In particular$,$ $\dim(\bar{A})=\dim A.$

\item [{\rm (iii)}] Let $A, B \subset \mathbb{R}^n$ be semi-algebraic sets. Then$,$
 $$\dim (A \cup B) = \max\{ \dim A, \dim B\}.$$
\end{enumerate}
\end{proposition}

Combining Theorems~2.4.4,~2.4.5 and Proposition~2.5.13 in \cite{Bochnak1998}, it follows that
\begin{proposition}\label{prop1}
Let $A$ be a semi-algebraic set of $\Bbb R^n.$ The following statements hold.
\begin{enumerate}
    \item [{\rm (i)}] $A$ has a finite number of connected components which are closed in $A.$
    \item [{\rm (ii)}] $A$ is connected if and only if it is path connected.
\end{enumerate}
\end{proposition}
Hence, in the rest of this paper, by saying that a semi-algebraic subset of $\RR^n$
is connected, we also mean that it is path connected.

\begin{theorem}\label{th::num}
	For any polynomials $f_1,\ldots,f_s,g_1,\ldots,g_l\in\RR[x]$ with
	degree bounded by $d\in\N$, the number of connected components of
	the semi-algebraic set
	\[
		S:=\{x\in\RR^n : f_1(x)=\cdots=f_s(x)=0,\ g_1(x)\neq 0,
		\ldots, g_l(x)\neq 0\},
	\]
is bounded from above by 
\[
	N(n,d,l):=\left\{
		\begin{aligned}
			&d(2d-1)^{n-1},& \text{if}\ l=0,\\
			&(ld+1)(2ld+1)^n,& \text{if}\ l>0.
		\end{aligned}
	\right.
\]
\end{theorem}
\begin{proof}
	See \cite[Proposition 3.9.4]{RASS} and the proof of \cite[Proposition 4.4.5]{RASS}. 
\end{proof}

Next we state a semi-algebraic version of Sard's theorem with the
parameter  in a simplified form which is sufficient for the
applications studied here. 
Recall that, for an open set $X\subset\RR^n$ and a $C^{\infty}$ map
$F: X\rightarrow\RR^m$, a point $y\in\RR^m$ is called a {\itshape
regular value} for
$F$ iff either $F^{-1}(y)=\emptyset$ or the derivative $DF(x): \RR^n
\rightarrow \RR^m$ is surjective at every point $x\in F^{-1}(y)$.
%A point $y \in Y$ that is not a regular value of $F$ is
%called a {\em critical value.}
%Given a differentiable map between
%manifolds $f \colon X \rightarrow Y,$ a point $y \in Y$ is called a
%{\em regular value} \index{regular value} for $f$ if either $f^{-1}(y)
%= \emptyset$ or the derivative map $Df(x) \colon T_xX \rightarrow
%T_yY$ is surjective at every point $x \in f^{-1}(y),$ where $T_x X$
%and $T_yY$ denote the tangent spaces of $X$ at $x$ and of $Y$ at $y,$
%respectively. A point $y \in Y$ that is not a regular value of $f$ is
%called a {\em critical value.} 
The following result is also called Thom's weak transversality theorem.

\begin{theorem}[Sard's theorem with parameter] \label{SardTheorem}
	Let $\mathscr{P}$ and $X$ be open semi-algebraic sets in $\RR^p$
	and $\RR^n$, respectively. Let $F: \mathscr{P}\times X \rightarrow
	\RR^m,\ (u,x)\rightarrow  F(u,x)$, be a semi-algebraic map of
	class $C^{\infty}$. If $y\in\RR^m$ is a regular value of $F$, then
	there exists an open and dense semi-algebraic subset $U$ in
	$\mathscr{P}$ such that, for each $u\in U$, $y$ is a regular value
	of the map $F_u : X\rightarrow \RR^m,\ x\rightarrow F(u,x)$.
%
%
%Let $f \colon X \times \mathscr{P} \rightarrow Y$ be a differentiable
%semi-algebraic map between semi-algebraic submanifolds.
%If $y \in Y$ is a regular value of $f,$ then there exists a
%semi-algebraic set $\Sigma \subset \mathscr{P}$
%of dimension smaller than the dimension of $\mathscr{P}$ such that$,$ for every $p
%\in \mathscr{P} \setminus \Sigma,$ $y$ is a regular value of the map
%$f_p \colon X \rightarrow Y, x \mapsto f(x, p).$
\end{theorem}
\begin{proof}
For a proof, we refer the reader to~\cite{DT} or
\cite[Theorem~1.10]{HaHV2017}.
\end{proof}

%\begin{lemma}[Curve Selection Lemma] \label{CurveSelectionLemma}
%Consider a semi-algebraic set $S \subset \mathbb{R}^n$ and a point $\bar{x} \in \mathbb{R}^n$ that is a cluster point of $S.$ Then there exists an analytic semi-algebraic curve $\phi \colon (-\epsilon, \epsilon) \to {\mathbb R}^n$ with $\phi(0) = \bar{x}$ and with $\phi(t) \in S$ for $t \in (0, \epsilon).$
%\end{lemma}
%
%\begin{lemma} [Monotonicity Lemma] \label{MonotonicityLemma}
%Let $\psi \colon [a, b] \rightarrow \mathbb{R}$ be a semi-algebraic function. Then there are finitely many points $a = t_0 < t_1 < \cdots < t_k = b$ such that the restriction of $\psi$ to each interval $(t_i, t_{i + 1})$ are analytic, and either constant, or strictly increasing or strictly decreasing.
%\end{lemma}

\subsection{Resultants and Discriminants}
Let us first review some elementary background about {\itshape resultants} and
{\itshape discriminants}.  More details can be found in
\cite{Cox-Little-OShea:UAG2005,DRMD,NieDisNon}. 

Let $f_1,\ldots,f_n$ be
homogeneous polynomials in $\RR[x]$. The resultant
$\res(f_1,\ldots,f_n)$ is a polynomial in the coefficients of $f_1,\ldots,f_n$
satisfying 
\[
\res(f_1,\ldots,f_n)=0\quad\Leftrightarrow\quad\exists 0\neq u\in\CC^n, 
\ f_1(u)=\cdots=f_n(u)=0.
\]
Let $f_1,\ldots,f_m$ be homogeneous polynomials with $m<n$ and suppose
that at least one $\deg(f_i)>1$.
The discriminant for $f_1,\ldots,f_m$, denoted by $\Delta(f_1,\ldots,f_m)$, is 
a polynomial in the coefficients of $f_1,\ldots,f_m$ such that
\[
\Delta(f_1,\ldots,f_m)=0
\]
if and only if the polynomial system
\[
f_1(x)=\cdots=f_m(x)=0
\]
has a solution $0\neq u\in\CC^n$ such that the Jacobian matrix of
$f_1,\ldots,f_m$ does not have full rank.

The resultants and discriminants are also defined for inhomogeneous
polynomials. Let $f_0,f_1,\ldots,f_n$ be general polynomials in
$\RR[x]$. The resultant $\res(f_0,f_1,\ldots,f_n)$ is defined to be
$\res(\td{f}_0(\td{x}),\td{f_1}(\td{x}),\ldots,\td{f}_n(\td{x}))$
where each $\td{f}_i(\td{x}):=x_0^{\deg(f_i)}f(x/{x_0})$ is the
homogenization of $f_i(x)$ in $\td{x}:=(x_0,x_1,\ldots,x_n)$. Clearly,
if $\res(f_0,f_1,\ldots,f_n)\neq 0$, then 
\[
	f_1(x)=\cdots=f_n(x)=0
\]
has no solution in $\CC^n$. Let $f_1,\ldots,f_m$ be general
polynomials in $\RR[x]$ with $m\le n$. 
The discriminant $\Delta(f_1,\ldots,f_m)$ is
defined to be $\Delta(\td{f_1}(\td{x}),\ldots,\td{f}_m(\td{x}))$. If
$\Delta(f_1,\ldots,f_m)\neq 0$, then it can be proved by Euler's
formula that the polynomial system
\[
f_1(x)=\cdots=f_m(x)=0
\]
has no solution $0\neq u\in\CC^n$ such that the Jacobian matrix of
$f_1,\ldots,f_m$ does not have full rank (c.f. \cite{NieDisNon}).

\subsection{Subdifferentials and nonsmooth slope}
Now we recall some notation and properties of subdifferential, which
will be used in this paper. The following materials and more details
can be found in the comprehensive texts
\cite{Bounkhel2012,Clarke1990,Mordukhovich2006,Rockafellar98} about
nonsmooth analysis.

\begin{definition}{\rm
Let $f: \RR^n \rightarrow \RR$ be a locally Lipschitz function.
\begin{enumerate}[\upshape (i)]
	\item The {\itshape generalized directional derivative} (also
		known as {\itshape Clarke directional derivative}) of $f$ at
		$\bar{x}\in\RR^n$ in the direction $v\in\RR^n$, denoted by
		$f^0(\bar{x};v)$, is given by  
		\[
			f^0(\bar{x};v):=\limsup_{{x\rightarrow\bar{x}}\atop{t\downarrow
			0}} \frac{f(x+tv)-f(x)}{t}.
		\]
	\item The {\itshape generalized gradient} ({\itshape Clarke
		subdifferential}) of $f$ at $\bar{x}$, denoted by $\partial^{\circ}
		f(\bar{x})$, is defined as
		\[
			\partial^{\circ} f(\bar{x}):=\{\zeta\in\RR^n : \langle \zeta,
				v\rangle\le f^0(\bar{x};v),\ \forall v\in\RR^n\}.
		\]
\end{enumerate}}
\end{definition}
There are many other concepts of subdifferentiability for nonconvex
functions, like the {\itshape Fr{\'e}chet subdifferential}, the
{\itshape limiting subdifferential} and so on. Note that these sets of
subdifferential coincide for any convex continuous function.
Therefore, we have 
\begin{example}\label{exam:1}
	For each $x^{0}\in\RR^n,$ we have
\[
	\partial^{\circ} (\Vert\cdot -x^{0}\Vert)(x)=\left\{
		\begin{aligned}
			&\frac{x-x^{0}}{\Vert x-x^{0}\Vert}\quad & \text{if}\
			x\neq x^{0},\\
			&\bar{\mathbb{B}}\quad & \text{otherwise},
		\end{aligned}
		\right.
\]
where $\bar{\mathbb{B}}$ denotes the closed unit ball centered at the
origin in $\RR^n$. 
\end{example}

The following properties of Clarke subdifferential will be used in our
arguments.

\begin{proposition}\label{prop::clarke}
Let $f: \RR^n \rightarrow \RR$ be a locally Lipschitz function$,$ then the
following statements are true.	
\begin{enumerate}[\upshape (i)]
	\item If $x^{0}$ is  a local minimizer of $f$, then
		$0\in\partial^{\circ} f(x^{0})$.
	\item Let $g: \RR^n \rightarrow \RR$ be a locally Lipschitz
		function$,$ then
		\[
			\partial^{\circ} (f+g)(x)\subset \partial^{\circ}
			f(x)+\partial^{\circ} g(x),\quad\text{for all}\ x\in\RR^n.
		\]
	\item Let $x^{0}\in\RR^n$ be such that $f(x^{0})>0,$ then for any
		$\rho>0,$
		\[
			\partial^{\circ} f^\rho(x^{0})=\rho[f(x^{0})]^{\rho-1}\partial^{\circ}
			f(x^{0}).
		\]
\end{enumerate}
\end{proposition}

\begin{remark}\label{rk::pp}{\rm
	Note that the properties in Proposition \ref{prop::clarke} also hold for
	the limiting subdifferential {\upshape (c.f.
		\cite{Mordukhovich2006})}.}
	\end{remark}

\begin{definition}\label{def::slope}{\rm
	Let $f: \RR^n \rightarrow \RR$ be a locally Lipschitz function.
	We define the non-smooth slope of $f$ at $x\in\mathbb{R}^n$ by
\[
	\mathfrak{m}^{\circ}_f(x):=\inf\{\Vert \zeta\Vert : \zeta\in	\partial^{\circ}
		f(x)\}.
\]
}
\end{definition}

\begin{remark}{\rm
	 The non-smooth slope of $f$ is defined
	using limiting subdifferential in \cite{DHP2017,LMP2015}, by
	which some results about non-smooth {\L}ojasiewicz's inequality
	and error bounds for the maximum functions of finitely many
polynomials are derived.   }
\end{remark}

To end this section, we recall a classic theorem in analysis which
states that we can find a ``minimizing sequence'' for a continuous
function $f$ which is bounded from below. 
\begin{theorem}[Ekeland Variational Principle]\label{th::ekeland}
	Let $f : \RR^n \rightarrow \RR$ be a continuous function$,$ bounded
	from below. Let $\varepsilon>0$ and $x^{0}\in\RR^n$ be such that
	\[
		\inf_{x\in\RR^n} f(x)\le f(x^{0}) \le \inf_{x\in\RR^n}
		f(x)+\varepsilon.
	\]
	Then for any $\lambda>0,$ there exists some point $y^{0}\in\RR^n$
	such that 
	\[
		\begin{aligned}
			&f(y^{0})\le f(x^{0}),\\
			&\Vert y^{0}-x^{0}\Vert\le\lambda,\\
			f(y^{0})\le f(x)&+\frac{\varepsilon}{\lambda}\Vert
			x-y^{0}\Vert\quad \text{for all}\ x\in\RR^n.
		\end{aligned}
	\]
\end{theorem}

\section{Continuous selections of polynomial functions}\label{sec::CSS}
In this section, we will give some formal definitions and obtain some
basic properties about continuous selections of polynomial (or more
generally, semi-algebraic) functions. 

\begin{definition}\label{def::ppf}{\rm
		For given subsets $X\subseteq\wt{X}\subseteq\RR^n$ and $r$ 
		continuous functions $f_1,\ldots,f_r: \wt{X}\rightarrow \RR$,
		we say a function $f \colon X \rightarrow
		\mathbb{R}$ a {\it continuous selection} of $\{f_1,\ldots,f_r\}$ if
		$f$ is
		continuous and $f(x) \in \{f_1(x), \ldots, f_r(x)\}$ for all
		$x \in X.$ We call $\{f_1,\ldots,f_r\}$ the set of {\itshape
		selection functions} of $f$.
		We denote by $\mc(f_1,\ldots,f_r,X)$ the set of all
		continuous selections of $\{f_1,\ldots,f_r\}$ with the domain
		$X\subseteq\RR^n$. If $X=\RR^n$, we use the
		notation $\mc(f_1,\ldots,f_r)$ for simplicity. We call
		$I(f,x):=\{i\mid f_i(x)=f(x)\}$ the active index set of $f$ at a
		point $x\in X$. 
%
%function of degree $d$ if $f$ is continuous and
%there exist finitely many polynomial functions $f_j \colon
%\mathbb{R}^n \rightarrow \mathbb{R}, j = 1, \ldots, r,$ with $d =
%\max_{j = 1, \ldots, r} \deg f_j$ such that $f(x) \in \{f_1(x),
%\ldots, f_r(x)\}$ for all $x \in X.$ 
%For short, we say $f  \colon X \subset \mathbb{R}^n \rightarrow
%\mathbb{R}$ is a CPP function of degree $d$ with a selection $\{f_1, \ldots, f_r\}$.
}\end{definition}
Obviously, the set $\mc(f_1,\ldots,f_r,X)$ contains various types of
functions composed by $\{f_1,\ldots,f_r\}$, among which are the
typical examples of the maximum and minimum functions occurring in
optimization
%In particular, if for each $i=1,\ldots,r$, $f_i(x)$ is
%quadratic and the piece $\{x\in\RR^n : f(x)=f_i(x)\}$ is a polyhedron,
%then the element in $\mc(f_1,\ldots,f_r,X)$ is called the piecewise
%linear-quadratic function.
%
%Two typical examples of continuous selection of $\{f_1,\ldots,f_r\}$
%are
\[
	f_{\max}(x):=\max\{f_1(x),\ldots,f_r(x)\}\quad\text{and}\quad
	f_{\min}(x):=\min\{f_1(x),\ldots,f_r(x)\}.
\]
More generally, the set $\mc(f_1,\ldots,f_r,X)$ contains the following
max-min type functions
\begin{equation}\label{eq::maxmin}
	\max_{i\in\{1,\ldots,s\}}\min_{j\in J_i}
	f_j(x)\quad\text{and}\quad
	\min_{i\in\{1,\ldots,s\}}\max_{j\in J_i} f_j(x),
\end{equation}
where each $J_i\subseteq\{1,\ldots,r\}$. In fact, every
function which is representable by a formula involving
$\{f_1,\ldots,f_r\}$ together with a finite number of maximum or
minimum operations can be written as a max-min type function (c.f.
\cite{BARTELS1995385}). Conversely, if each $f_i$ is affine, then
it is shown in \cite[Corollary 2.1]{BARTELS1995385} that every
function in $\mc(f_1,\ldots,f_r)$ can be expressed in the
max-min type. 
%Moreover, by means of \cite[Corollary
%2.3]{BARTELS1995385}, if $f_1,\ldots,f_r\in\RR[x]$ with generic
%coefficients, then every $f\in\mc(f_1,\ldots,f_r)$ is locally
%representable as a max-min type selection of the functions $f_i(x),
%i\in I(f,x^0)$.
\begin{proposition}\label{prop::Lip}
If $X$ is open and each $f_j$ is $C^1$-function$,$ $j = 1, \ldots, r,$ then each $f\in
\mc(f_1,\ldots,f_r,X)$ is locally Lipschitz. In this case$,$ for any $x\in
X,$
\begin{equation}\label{eq::clarkesub}
	\partial^{\circ} f(x)=\co\left\{\lim_{y\rightarrow x}\nabla f_i(y) :\
i\in I(f,x)\right\}=\co\{\nabla f_i(x) :\ i\in I(f,x)\},
\end{equation}
and hence
\begin{equation}\label{eq::mf}
		\mathfrak{m}^{\circ}_f(x)
		=\inf\left\{\Big\Vert \sum_{i\in I(f,x)}\mu_i\nabla f_i(x)\Big\Vert :
		\mu_i\ge 0, \sum_{i\in I(f,x)}\mu_i=1\right\}.
	\end{equation}
\end{proposition}
\begin{proof}
Clearly, it is not difficult to verify that every $C^1$-function is locally Lipschitz continuous.
The results then follow by \cite[Corollary 4.1.1]{ScholtesBook} and \cite[Theorem
2.5.1]{Clarke1990}. 
\end{proof}
Consequently, if $X$ is open and each $f_j$ is $C^1$-function, $j = 1, \ldots, r,$ we
can define the critical point of $f$ in the following way.
\begin{definition}{\rm
	Let $f\in\mc(f_1,\ldots,f_r,X),$ where $X$ is open and each $f_j$
	is $C^1$-function, $j = 1, \ldots, r,$ a point $x^0\in X$ is called a {\it critical
	point} of $f$ if $0\in\partial^{\circ} f(x^0)$, i.e., there exists
	a tuple $(\mu_i\in\RR,\ i\in I(f,x^0))$ such that 
	\begin{equation}\label{eq::lambda}
		\mu_i\ge 0,\  i\in
	I(f,x^0),\quad \sum_{i\in I(f,x^0)}\mu_i=1\quad\text{and}\quad\sum_{i\in
		I(f,x^0)}\mu_i\nabla f_i(x^0)=0.
	\end{equation}
	If each $\mu_i>0,\ i\in I(f,x^0))$ for any tuple
	$(\mu_i\in\RR,\ i\in I(f,x^0))$ satisfying \eqref{eq::lambda}, we
	say the {\it strict complementarity} holds for the critical point
	$x^0$.
}\end{definition}
The first order necessary optimality condition states that a local
minimizer of $f\in\mc(f_1,\ldots,f_r,X)$ must be a critical point (c.f. \cite[Theorem
3.1]{Womersley1982}). 

 We denote by $\cri(f,X)$ the set of all critical points of $f$ on
 $X$ and by  $\cri(\mc(f_1,\ldots,f_r,X))$ the set of all critical
	points of all continuous selections in $\mc(f_1,\ldots,f_r,X)$,
	i.e.,
	\[
		\cri(\mc(f_1,\ldots,f_r, X))=\{x^0\in X: \exists
			f\in\mc(f_1,\ldots,f_r, X)\ \text{such that}\ 0\in\partial^{\circ}
		f(x^0)\}.
	\]
For simplicity, we adopt the notation $\cri(f)$ and
$\cri(\mc(f_1,\ldots,f_r))$ when $X=\RR^n$.

Now let us see some favorable properties enjoyed by continuous
selections of continuous semi-algebraic functions $\{f_1,\ldots,f_r\}$
on a semi-algebraic set $X$.

\begin{theorem}\label{th::Lip}
Let $X \subset \mathbb{R}^n$ be a semi-algebraic set and $f_1, \ldots,
f_r \colon X \rightarrow \mathbb{R}$ be continuous semi-algebraic
functions. Then$,$ $\mc(f_1,\ldots,f_r,X)$ is a finite set and each
$f\in \mc(f_1,\ldots,f_r,X)$ is semi-algebraic.
%Let $f \colon X \rightarrow \mathbb{R}$ be a continuous
%selection of $\{f_1,\ldots,f_r\}$, then 
%\begin{enumerate}[\upshape (i)]
%	\item $\mc(f_1,\ldots,f_r,X)$ is a finite set and each $f\in
%		\mc(f_1,\ldots,f_r,X)$ is semi-algebraic;
%	\item If $X$ is open and each $f_j$ is locally Lipschitz, then
%		each $f\in \mc(f_1,\ldots,f_r,X)$ is locally Lipschitz.
%\end{enumerate}
\end{theorem}
\begin{proof}
	It is clear for the case $r=1$. Assume that the conclusion
	holds for $r=k$, then we prove that it is also true for $r=k+1$. 
	Then, the conclusion follows by induction on $r$.
%	By definition, we need to show that the graph $\mathscr{G}:=\{(x,y)\in X\times
%		\mathbb{R} : y=f(x)\}$ is semi-algebraic subset of
%		$\mathbb{R}^n\times\mathbb{R}$. 

		Let $A=\cup_{j=1}^k\{x\in X : f_j(x)=f_{k+1}(x)\}$.
		Clearly, $A$ is a semi-algebraic set. For any
		$f\in\mc(f_1,\ldots,f_{k+1},X)$, its restrictions on $A$ and
		$X\setminus A$ are functions in $\mc(f_1,\ldots,f_{k+1},A)$ and
		$\mc(f_1,\ldots,f_{k+1},X\setminus A)$, respectively.
		In the following, we only need to prove that both
		$\mc(f_1,\ldots,f_{k+1},A)$ and
		$\mc(f_1,\ldots,f_{k+1},X\setminus A)$ are finite sets and
		each function in $\mc(f_1,\ldots,f_{k+1},A)$ and
		$\mc(f_1,\ldots,f_{k+1},X\setminus A)$ is semi-algebraic.
		
	For any $f\in\mc(f_1,\ldots,f_{k+1},A)$, by the definition of
	$A$, $f(x)\in\{f_1(x),\ldots,f_k(x)\}$ for all $x\in A$.
	Therefore, it holds that
	$\mc(f_1,\ldots,f_{k+1},A)=\mc(f_1,\ldots,f_k,A)$. Then, by
	induction, $\mc(f_1,\ldots,f_{k+1},A)$ is a finite set and each
	$f\in\mc(f_1,\ldots,f_{k+1},A)$ is semi-algebraic.

	Since $X\setminus A$ is semi-algebraic, it has finitely many
	semi-algebraically (path) connected components, say
	$D_1,\ldots,D_s$. Now, it suffices to prove that 
	$\mc(f_1,\ldots,f_{k+1},D_i)$ is a finite set and each $f\in
	\mc(f_1,\ldots,f_{k+1},D_i)$ is semi-algebraic for every
	$i=1,\ldots,s$. To this end, we show that for each $f\in
	\mc(f_1,\ldots,f_{k+1},D_i)$, either $f\in
	\mc(f_1,\ldots,f_k,D_i)$ or $f$ is the restriction of $f_{k+1}$ on
	$D_i$. Then, the
	conclusion follows by induction.  To the contrary, suppose that
	there exist $u,v\in D_i$ such that
	$f(u)\in\{f_1(u),\ldots,f_k(u)\}$ and $f(v)=f_{k+1}(v)$. 
	Since $D_i$ is path connected, there exists a continuous curve
	$\phi: [0,1]\rightarrow D_i$ such that $\phi(0)=u$ and $\phi(1)=v$.
	Let 
	\[
		\td{\tau}:=\sup\{\tau\in[0,1]:
		f(\phi(t))\in\{f_1(\phi(t)),\ldots,f_k(\phi(t))\}\ \text{for
		all}\ t\in[0,\tau]\}.
	\]
	By the continuity, we have
	$f(\phi(\td{\tau}))\in\{f_1(\phi(\td{\tau})),\ldots,f_k(\phi(\td{\tau}))\}$
	and $f(\phi(\td{\tau}))=f_{k+1}(\phi(\td{\tau}))$. It implies that
	$\phi(\td{\tau})\in A$, a contradiction.
%
%
%		and $f_j$'s are continuous on $X$, each $A_j$ is closed in the
%	relative topology of $X$ induced by the Euclidean topology of
%	$\mathbb{R}^n$. Clearly, the set 
%	\[
%		X\setminus\bigcup_{1\le j_1\neq j_2\le k+1}(A_{j_1}\cap A_{j_2})
%	\]
%is semi-algebraic and hence has finitely many semi-algebraically
%connected components, say $D_1,\ldots,D_s$.  For each $i\in\{1,\ldots,s\}$,
%we have $D_i=\cup_{j=1}^{k+1}(D_i\cap A_j)$. Note that the sets
%$D_i\cap A_j$, $j=1,\ldots,k+1$ are closed in $D_i$ and disjoint.
%Since $D_i$ is semi-algebraically connected, there is a unique
%$j_0\in\{1,\ldots,k+1\}$ such that $D_i=D_i\cap A_{j_0}$, i.e.,
%$f(x)=f_{j_0}(x)$ for all $x\in
%D_i$. Hence, $\mathscr{G}\cap (D_i\times\mathbb{R})$ is a semi-algebraic
%subset of $\mathbb{R}^n\times\mathbb{R}$ for each $i$. For $1\le
%j_1\neq j_2\le k+1$, the restriction $f: A_{j_1}\cap A_{j_2}
%\rightarrow \mathbb{R}$ is a CPP function with a selection
%$\{f_1,\ldots,f_{k+1}\}\setminus\{f_{j_1}\}$. Then by assumption,
%$\mathscr{G}\cap( A_{j_1}\cap A_{j_2}\times\mathbb{R})$ is a semi-algebraic
%subset of $\mathbb{R}^n\times\mathbb{R}$. Now we have proved that
%$\mathscr{G}$ is a finite union of semi-algebraic subsets of
%$\mathbb{R}^n\times\mathbb{R}$ which means that $f$ is semi-algebraic.
%
\end{proof}

\begin{remark}{\rm
For arbitrary subset $X\subseteq\RR^n$ and continuous functions
$f_1,\ldots,f_r: \RR^n\rightarrow\RR$, the set $\mc(f_1,\ldots,f_r,X)$
is not necessarily finite. For example, it is clear that the set
$\mc(\sin x,\cos x,\RR)$ has infinitely many functions.}
\end{remark}

\begin{theorem}\label{th::minimizers}
Let $X \subset \mathbb{R}^n$ be a semi-algebraic set and $f_1, \ldots,
f_r \colon X \rightarrow \mathbb{R}$ be continuous semi-algebraic
functions. For any $f\in\mc(f_1,\ldots,f_r,X),$ the following holds$:$
\begin{enumerate}[\upshape (i)]
	\item The set of local $($resp.$,$ isolated local$,$ strictly local$)$
		minimizers of $f$ is semi-algebraic$;$ 
	\item The set of isolated local minimizers of $f$ coincides with
		its set of strictly local minimizers and both are finite.
\end{enumerate}
\end{theorem}
\begin{proof}
	It is a consequence of Theorem \ref{th::Lip} and the following Propositions
	\ref{Claim1} and \ref{Claim2}. 
\end{proof}

\begin{proposition}\label{Claim1}
Let $X \subset \mathbb{R}^n$ be a semi-algebraic set and $f \colon X
\rightarrow \mathbb{R}$ be a semi-algebraic function. 
Then the set of $($strictly$)$ local minimizers of $f$ is semi-algebraic.
%, and so is a finite set.
\end{proposition}
\begin{proof}
	We only prove the statement for local minimizers and similar
	arguments hold for strictly local minimizers.

Let $A$ be the set of local minimizers of $f.$ By definition, we can write
\begin{eqnarray*}
A &=& \{x \in X \colon \exists \delta > 0 \textrm{ such that } f(y) \ge
f(x) \textrm{ for all } y \in X, 0 < \|y - x\| < \delta\}.
\end{eqnarray*}
Clearly, $A = \pi_1(B)$ where $\pi_1 \colon \mathbb{R}^n \times
\mathbb{R} \rightarrow \mathbb{R}^n, (x, \delta) \mapsto x,$ and 
\begin{eqnarray*}
B &:=& \{(x, \delta) \in X \times \mathbb{R}_{> 0} \colon f(y) \ge f(x)
\textrm{ for all } y \in X, 0 < \|y - x\| < \delta\}.
\end{eqnarray*}
Let $C := X \times \mathbb{R}_{> 0} \setminus B.$ We can write
\begin{eqnarray*}
C &:=& \{(x, \delta) \in X \times \mathbb{R}_{> 0} \colon \exists y \in
X, 0 < \|y - x\| < \delta, f(y) < f(x)\}.
\end{eqnarray*}
Clearly, $C = \pi_2(D)$ where $\pi_2 \colon \mathbb{R}^n \times
\mathbb{R} \times \mathbb{R}^n \rightarrow \mathbb{R}^n \times
\mathbb{R}, (x, \delta, y) \mapsto (x, \delta),$ and 
\begin{eqnarray*}
D &:=& \{(x, \delta, y) \in X \times \mathbb{R}_{> 0} \times X \colon 0
< \|y - x\|, \|y - x\| - \delta < 0, f(y) - f(x) < 0\}.
\end{eqnarray*}
Note that the sets $X \times \mathbb{R}_{> 0}$ and $X \times
\mathbb{R}_{> 0} \times X$ are semi-algebraic (see
Proposition~\ref{property of SA}(iii)), and the functions $(x, \delta, y)
\mapsto \|y - x\|, (x, \delta, y) \mapsto \|y - x \| - \delta,$ and
$(x, \delta, y) \mapsto f(y) - f(x)$ are semi-algebraic. Hence $D$ is
a semi-algebraic set. By the Tarski--Seidenberg Theorem, $C =
\pi_2(D)$ is a semi-algebraic set. By Proposition~\ref{property of
SA}(iii), $B = X \times \mathbb{R}_{> 0} \setminus C$ is a
semi-algebraic set. By the Tarski--Seidenberg Theorem again, 
$A = \pi_1(B)$ is a semi-algebraic set. 
\end{proof}

%By a similar argument, it is not hard to see that the set of isolated local minimizers of $f$ is semi-algebraic, and so is a finite set.

\begin{proposition}[compare \cite{Absil2006,Cui2018-2,Cui2018-1}] \label{Claim2}
Let $X \subset \mathbb{R}^n$ be a semi-algebraic set and $f \colon X
\rightarrow \mathbb{R}$ be a semi-algebraic function
which is continuous around a point $\bar{x} \in X.$ Then, $\bar{x}$ is
an isolated local minimizer of $f$ if and only if $\bar{x}$ is a
strictly local minimizer of $f.$ Consequently$,$ the set of isolated
$($strictly$)$ local minimizers is finite.
%
%the following
%statements are equivalent:
%\begin{enumerate}[{\rm (i)}]
%\item $\bar{x}$ is a local minimizer and isolated critical point of $f;$
%\item $\bar{x}$ is an isolated local minimizer of $f;$
%\item $\bar{x}$ is a strict local minimizer of $f.$
%\end{enumerate}
\end{proposition}
\begin{proof}
%(i) $\Leftrightarrow$ (ii). This is a direct consequence of the \L ojasiewicz gradient inequality. TO BE WRITTEN.
%
%(ii) $\Rightarrow$ (iii). 
% $\Rightarrow$ (ii). 
	An isolated local minimizer of $f$ is clearly a strictly local
	minimizer. Now we prove the other direction. 
Suppose to the contrary that there exists a
sequence $x^{(k)} \in X, k \in \mathbb{N},$ with $x^{(k)} \ne \bar{x}$ and
$x^{(k)} \to \bar{x}$ such that for each $k,$ $x^{(k)}$ is a local minimizer
 of $f.$ Let $A$ be the set of local minimizers of $f.$ Then, by
 Proposition~\ref{Claim1}, we can see that $A$ is a semi-algebraic set. 
By the Curve Selection Lemma~\ref{CurveSelectionLemma} there exists an
analytic semi-algebraic curve $\phi \colon (-\epsilon, \epsilon) \to
\mathbb{R}^n$ such that $\phi(0) = \bar{x}$ and $\phi(t) \in A$ for
all $t \in (0, \epsilon).$ By the Monotonicity Lemma~\ref{MonotonicityLemma}, we can
assume that the semi-algebraic function $\psi \colon [0,
	\epsilon) \to \mathbb{R}, t \mapsto f \circ \phi (t),$ is analytic
	on $(0, \epsilon)$ and  is either constant, or strictly increasing
	or strictly decreasing. Note that for each $t \in (0, \epsilon),$
	we have $t$ is a local minimizer of $\psi.$ Hence $\psi$ is
	constant on $(0, \epsilon).$ Since $f$ is continuous at $\bar{x},$
	the function $\psi$ is continuous at $t = 0.$ Consequently, 
$$f \circ \phi (t) = f \circ \phi (0) = f(\bar{x})$$ 
for all $t \in [0, \epsilon) ,$ which contradicts our assumption that
	$\bar{x}$ is a strictly local minimizer.

Since a semi-algebraic set has finitely many connected components, by
Proposition \ref{Claim1}, the set of isolated (strictly) local
minimizers of $f$ is finite.
\end{proof}

In the rest of this paper, we mainly study the set of continuous selections of
polynomial functions, i.e., $X=\RR^n$ and $f_1,\ldots,f_r\in\RR[x]$.
We also call such functions CSP functions, for short, selected from
$\{f_1,\ldots,f_r\}$. 
An example of CSP functions is the so-called piecewise
linear-quadratic function. Precisely, if for each $i=1,\ldots,r$,
$f_i(x)$ is quadratic and the piece $\{x\in\RR^n : f(x)=f_i(x)\}$ is a
polyhedron, then the element in $\mc(f_1,\ldots,f_r)$ is called the
piecewise linear-quadratic function, which is investigated in
\cite{Cui2018-1}.
%EXAMPLES
%In particular, each PLQ function is a continuous selection of linear
%and quadratic functions. 

By Proposition \ref{prop::Lip}, a CSP function $f$ is locally
Lipschitz. Therefore, the Clarke subdifferential
and non-smooth slope of a CSP function $f$ at any $x\in\mathbb{R}^n$ are of the forms
\eqref{eq::clarkesub} and \eqref{eq::mf}, respectively. 

\begin{proposition}\label{prop::cri}
	For any $f_1,\ldots,f_r\in\RR[x]$ with degrees bounded by $d,$ the
	set $\cri(\mc(f_1,\ldots,f_r))$ is semi-algebraic and has at most 
	\begin{equation}\label{eq::B0}
	B_0(n,d,r):=(d+1)(2d+1)^{n+r-1}+\sum_{s=1}^{r-1}\binom{r}{s}\left[(r-s)(d+1)+1\right]\left[2(r-s)(d+1)+1\right]^{n+s}
\end{equation}
connected components.
\end{proposition}
\begin{proof}
%	Fix a subset $J\subseteq\{1,\ldots,r\}$ with $\# J=s\ge 1$. For
%	simplicity, let $J=\{1,\ldots,s\}$ and 
%	$u^J:=(u^{(1)}, \ldots, u^{(s)})\in\RR^{s\times n(d)}$
%	for any $u\in\RR^{r\times n(d)}$. Define
	For any nonempty subset
	$J=\{j_1,\ldots,j_s\}\subseteq\{1,\ldots,r\}$, define 
	\begin{equation}\label{eq::cu}
%		\mathcal{C}_J(u):=\left\{(x,\lambda)\in\RR^n\times\RR^{s} :
%		\left\{
%	\begin{aligned}
%	&	\sum_{j=1}^s\lambda_j^2\nabla_x f_j(x,u^{(j)})=0,\ \ 
%		\sum_{j=1}^s\lambda_j^2=1,\\
%		&f_j(x,u^{(2)})-f_1(x,u^{(1)})=0, \ j=2,\ldots,s,\\
%		&f_k(x,u^{(2)})-f_1(x,u^{(1)})\neq 0,\  k=s+1,\ldots,r
%	\end{aligned}
%	\right.
%	\right\}
		\mathcal{C}_J:=\left\{(x,\lambda)\in\RR^n\times\RR^{s} :
		\left\{
	\begin{aligned}
		&	\sum_{k=1}^s\lambda_k^2\nabla f_{j_k}(x)=0,\ \ 
		\sum_{k=1}^s\lambda_k^2=1,\\
		&f_{j_k}(x)-f_{j_1}(x)=0, \ k=2,\ldots,s,\\
		&f_i(x)-f_{j_1}(x)\neq 0,\ i\not\in J 
	\end{aligned}
	\right.
	\right\}.
\end{equation}
%For any subset $J\subseteq\{1,\ldots,r\}$ with $\# J\ge 1$, 
Let $\pi_J : \RR^n\times \RR^s\rightarrow \RR^n$ be the
projection on the first $n$ coordinates.
For any $x^0\in\cri(\mc(f_1,\ldots,f_r))$,
%with $u\in\RR^{r\times n(d)}$,
there exists a continuous selection
$f\in\mc(f_1,\ldots,f_r)$ such that $x^0$ is a critical point of $f$. Then,
$x^0\in\pi_{I(f,x^0)}(\mathcal{C}_{I(f,x^0)})$. Conversely, for
any $x^0\in\pi_J(\mathcal{C}_J)$ with nonempty subset
$J\subseteq\{1,\ldots,r\}$,
%$u\in\RR^{r\times n(d)}$ and $\# J\ge 1$, 
it is clear that $x^0$ is critical point of every  
$f\in\mc(f_1,\ldots,f_r)$ with $I(f,x^0)=J$. Therefore,  
\[
	\cri(\mc(f_1,\ldots,f_r))=\bigcup_{J} \pi_J(\mathcal{C}_J),
\]
where the union is taken over all nonempty subsets $J$ of
$\{1,\ldots,r\}$. Hence, by Proposition \ref{property of SA}, the set
$\cri(\mc(f_1,\ldots,f_r))$ is semi-algebraic. 
For each nonempty subset $J$ with $\# J=s$, by Theorem \ref{th::num}, the number
of connected component of $\mathcal{C}_J$ is bounded from above by 
\[
	\left\{
	\begin{aligned}
&(d+1)(2d+1)^{n+r-1}, & \text{if}\ s=r,\\
		&\left[(r-s)(d+1)+1\right]\left[2(r-s)(d+1)+1\right]^{n+s},&
		\text{if}\ s<r.\\
	\end{aligned}
	\right.
\]
Then, the conclusion follows. 
\end{proof}
For arbitrary continuous functions $f_1,\ldots,f_r: \RR^n\rightarrow\RR$, 
the above Proposition~\ref{prop::cri} will no longer hold. In
particular, the set $\cri(\mc(f_1,\ldots,f_r))$ does not
necessarily have finitely many connected components. For example,
consider the set $\mc(\sin x,\cos x,\RR)$.

\section{Genericity for continuous selections of polynomial
functions}\label{sec::genericity}

For $f_1,\ldots,f_r\in\RR[x]$, some generic properties about the
set $\cri(\mc(f_1,\ldots,f_r))$ will be established in this section.

For any positive integers $n$ and $d$, let
\[
	n(d):=\#\{\alpha:=(\alpha_1,\ldots,\alpha_n)\in\N^n : |\alpha|\le
	d\}, 
\]
where $|\alpha|:=\alpha_1+\cdots+\alpha_n$. 
Corresponding to the set of lexicographically ordered monomials
$x^\alpha, |\alpha|\le d$, we define for the variables
$x=(x_1,\ldots,x_n)$ a $n(d)$-component vector
\[
	\ve(x):=(1, x_1, \ldots, x_n, x_1^2, x_1x_2, \ldots, x_1x_n, \ldots, x_1^d, \ldots, x_n^d)^T, 
\]
which is known as the canonical basis of $\RR[x]$ with degree at most $d.$
For each parameter $u:=(u^{(1)},\ldots,u^{(r)})\in\RR^{r\times n(d)}$
where $u^{(i)}:=(u^{(i)}_\alpha)_{|\alpha|\le d}\in\RR^{n(d)}$, let
$f_i(x,u^{(i)})=\ve(x)^Tu^{(i)}$, $i=1,\ldots,r$, and
$F(x,u):=(f_1(x,u^{(1)}),\ldots,f_r(x,u^{(r)}))$. 

In this section, as $u^{(1)},\ldots,u^{(r)}$ are sometimes treated as
variables, we denote by $\nabla_x f_i(x, u^{(i)})$ the gradient
(column) vector of $f_i(x, u^{(i)})$ with respect to the variables
$x_1,\ldots,x_n$.

\begin{proposition}\label{prop::u1}
	There exists an open and dense semi-algebraic set
	$\mathscr{U}_1$ in $\RR^{r\times n(d)}$ such that for any
	$u\in\mathscr{U}_1$ and any $f\in\mc(F(x,u))$, we have
	$\#I(f,x)\le n+1$ at any $x\in\RR^n$. 
\end{proposition}
\begin{proof}
	Fix a subset $J:=\{j_1,\ldots,j_{n+2}\}\subseteq\{1,\ldots,r\}$
	with $j_1<\cdots<j_{n+2}$. Let
	$u^J:=(u^{(j_1)},\ldots,u^{(j_{n+2})})\in\RR^{(n+2)\times n(d)}$
	for any $u\in\RR^{r\times n(d)}$. Define
	\[
		R_J(u^J):=\res(f_{j_2}(x,u^{(j_2)})-f_{j_1}(x,u^{(j_1)}),\ldots,
		f_{j_{n+2}}(x,u^{(j_{n+2})})-f_{j_1}(x,u^{(j_1)}))\in\RR[u^J]\subset\RR[u],
	\]
	where $\res(\cdot,\ldots,\cdot)$ denotes the resultant of
	polynomials with respect to the variables $x$. Then, for any
	$u\in\RR^{r\times n(d)}$ with $R_J(u^J)\neq 0$, the polynomial system 
	\[
f_{j_2}(x,u^{(j_2)})-f_{j_1}(x,u^{(j_1)})=\cdots=f_{j_{n+2}}(x,u^{(j_{n+2})})-f_{j_1}(x,u^{(j_1)})=0
	\]
has no solutions in $\RR^n$. Let
\[
	\mathscr{U}_1:=\bigcap_{ {J\subseteq
		\{1,\ldots,r\}}\atop{\#J=n+2}}\{u\in\RR^{r\times n(d)} : R_J(u^J)\neq
	0\},
\]
which is an open and dense semi-algebraic set in $\RR^{r\times n(d)}$.
Clearly, for any
	$u\in\mathscr{U}_1$ and any $f\in\mc(F(x,u))$, %the active index set
	$\#I(f,x)\le n+1$ at any $x\in\RR^n$.
\end{proof}

\begin{proposition}\label{prop::u2}
	There exists an open and dense semi-algebraic set
	$\mathscr{U}_2$ in $\RR^{r\times n(d)}$ such that for any
	$u\in\mathscr{U}_2$ and any $f\in\mc(F(x,u)),$ the vectors
	$\nabla f_i(x^0), i\in I(f,x^0)$ are affinely independent for any
	critical point $x^0$ of $f,$ i.e.$,$ for every $i\in I(f,x^0),$ the
	vectors $\nabla_x f_{i'}(x^0,u^{(i')})-\nabla_x f_i(x^0,u^{(i)}),\ 
	i'\in I(f,x^0)\setminus\{i\},$ are linearly independent. Consequently$,$
	the tuple $(\mu_i, i\in I(f,x^0))$ satisfying \eqref{eq::lambda}
	is unique.
\end{proposition}
\begin{proof}
	Fix a subset $J:=\{j_1,\ldots,j_s\}\subseteq\{1,\ldots,r\}$	with
	$2\le s\le n+1$ and $j_1<\cdots<j_s$. Let
	$u^J:=(u^{(j_1)},\ldots,u^{(j_s)})\in\RR^{s\times n(d)}$
	for any $u\in\RR^{r\times n(d)}$. Now we define a polynomial
	$\Delta_J(u^J)\in\RR[u^J]\subset\RR[u]$ in the following way. If
	$\deg_x(f_{j_k}(x,u^{(j_k)})-f_{j_1}(x,u^{(j_1)}))\le 1$ for all
	$k=2,\ldots,s$, let $\Delta_J(u^J)$ be the sum of squares of all
	the maximal minors of the Jacobian matrix of
	$f_{j_k}(x,u^{(j_k)})-f_{j_1}(x,u^{(j_1)}), k=2,\ldots,s$, with
	respect to the variables $x$; otherwise, let 
	\[
		\Delta_J(u^J):=\Delta(f_{j_2}(x,u^{(j_2)})-f_{j_1}(x,u^{(j_1)}),\ldots,
		f_{j_s}(x,u^{(j_s)})-f_{j_1}(x,u^{(j_1)}))\in\RR[u^J]\subset\RR[u],
	\]
	where $\Delta(\cdot,\ldots,\cdot)$ denotes the discriminant of
	polynomials with respect to the variables $x$. Then, for any
	$u\in\RR^{r\times n(d)}$ with $\Delta_J(u^J)\neq 0$, the polynomial system 
	\[
		f_{j_2}(x,u^{(j_2)})-f_{j_1}(x,u^{(j_1)})=\cdots=f_{j_s}(x,u^{(j_s)})-f_{j_1}(x,u^{(j_1)})=0
	\]
has no solutions $x\in\RR^n$ such that the
vectors 
	\[
		\nabla_x(f_{j_2}(x,u^{(j_2)})-f_{j_1}(x,u^{(j_1)})),\ldots,
		\nabla_x(f_{j_s}(x,u^{(j_s)})-f_{j_1}(x,u^{(j_1)})),
	\]
%$\nabla f_{j'}(x^0,u^{j'})-\nabla f_j(x^0,u^{(j)}), j'\in
%J\setminus\{j\}$, 
	are linearly dependent. Let 
	%$\Delta_J(u^{J})=\prod_{j\in J}\Delta_{J,j}(u^J)$
\[
	\mathscr{U}_2:=\mathscr{U}_1\cap\left(\bigcap_{ {J\subseteq
		\{1,\ldots,r\}}\atop{2\le \#J\le n+1}}\{u\in\RR^{r\times n(d)} :
	\Delta_J(u^J)\neq 0\}\right),
\]
where  $\mathscr{U}_1$ is the open and dense semi-algebraic set in
$\RR^{r\times n(d)}$ in Proposition \ref{prop::u1}.
Clearly, $\mathscr{U}_2$ is an open and dense semi-algebraic set in
$\RR^{r\times n(d)}$.
It is straightforward to verify that for any
	$u\in\mathscr{U}_2$ and any $f\in\mc(F(x,u))$, the vectors
	$\nabla f_i(x^0), i\in I(f,x^0)$ are affinely independent for any
	critical point $x^0$ of $f$.

	If there are two tuples $(\mu_i, i\in I(f,x^0))$ and $(\eta_i,
	i\in I(f,x^0))$ satisfying \eqref{eq::lambda}, then for any $i\in
	I(f,x^0)$,
	\[
		\begin{aligned}
			&\nabla_x f_i(x^0,u^{(i)})+\sum_{{i'\in I(f,x^0)}\atop{i'\neq
			i}}\mu_{i'}\left(\nabla_x f_{i'}(x^0,u^{(i')})-\nabla_x
			f_i(x^0,u^{(i)})\right)=0,\\
			&\nabla_x f_i(x^0,u^{(i)})+\sum_{{i'\in I(f,x^0)}\atop{i'\neq
			i}}\eta_{i'}\left(\nabla_x f_{i'}(x^0,u^{(i')})-\nabla_x
			f_i(x^0,u^{(i)})\right)=0.
		\end{aligned}
	\]
We have 
\[
	\sum_{{i'\in I(f,x^0)}\atop{i'\neq i}}(\mu_{i'}-\eta_{i'})\left(\nabla_x f_{i'}(x^0,u^{(i')})-\nabla_x
			f_i(x^0,u^{(i)})\right)=0.
\]
Due to the linear independency, $\mu_{i'}=\eta_{i'}$ for all $i'\in I(f,x^0), i'\neq i$ and
clearly $\mu_i=\eta_i$. 
\end{proof}
Recall the max-min type selections in \eqref{eq::maxmin}. For
continuous selections of polynomial functions with generic
coefficients, we have the following local max-min representation at
their critical points.
\begin{corollary}
Let $\mathscr{U}_2$ be the open and dense semi-algebraic set in
$\RR^{r\times n(d)}$ in Proposition \ref{prop::u2}. 
For any $u\in\mathscr{U}_2,$ $f\in\mc(F(x,u))$ and 
$x^0\in\cri(f),$ $f$ is locally representable as a max-min type
selection of the functions $f_i(x),\ i\in I(f,x^0)$. 
\end{corollary}
\begin{proof}
It is a direct consequence of \cite[Corollary 2.3]{BARTELS1995385} and
Proposition \ref{prop::u2}.
\end{proof}

The following generic properties hold for the set of critical points
of all CSP functions selected from the same set of finitely many
polynomials.

\begin{theorem}\label{th::u3}
	There exists an open and dense semi-algebraic set
	$\mathscr{U}_3$ in $\RR^{r\times n(d)}$ such that for any
	$u\in\mathscr{U}_3$, 
%	\begin{enumerate}[\upshape (i)]
%		\item 
	the number of the points in $\cri(\mc(F(x,u)))$ is
			finite and bounded from above by $B_0(n,d,r)$ in
			\eqref{eq::B0}. 
			%has at most 
%\[
%	\wt{B}_0(n,d,r):=\left\{
%		\begin{aligned}
%			&	B_0(n,d,r), & \text{if}\ r\le n+1,\\
%			&
%			\sum_{s=1}^{n+1}\binom{r}{s}\left[(r-s)d+1\right]\left[2(r-s)d+1\right]^{n+s},
%			& \text{if}\ r>n+1
%		\end{aligned}
%		\right.
%\]
%			points 
%			for any $u\in\mathscr{U}_3$;
%		\item the set-valued map $u\mapsto
%			\cri(\mc(F(x,u)))$ is continuous on the set $\mathscr{U}_3$;
%		\end{enumerate}
			Moreover$,$ for any $u\in\mathscr{U}_3$, $f\in\mc(F(x,u))$
			and $x^0\in\cri(f),$
		\begin{enumerate}[\upshape (i)]
		\item  the
			strict complementarity holds for $x^0;$
		\item the system $\left(\sum_{i\in I(f,x^0)}\mu_i \nabla_x^2
			f_i(x^0,u^{(i)})\right)y=0$ where $\mu_i$ satisfying \eqref{eq::lambda} does
	not have a nonzero solution in the set
	\begin{equation}\label{eq::s1}
		\{y\in\RR^n : \nabla_x f_i(x^0, u^{(i)})^Ty=0, \ i\in
	I(f,x^0)\}. 
\end{equation}
%	as a critical point of $f$;
%\item for any $u\in\mathscr{U}_3$, $f\in\mc(F(x,u))$ and
%	$x^0\in\cri(f)$, the matrix $\sum_{i\in I(f,x^0)}\mu_i \nabla_x^2
%	f_i(x^0,u^{(i)})$ where $\mu_i$ satisfying \eqref{eq::lambda} does
%	not have a nonzero solution in the set
%	\[
%		\{y\in\RR^n : \nabla_x f(x^0, u^{(i)})^Ty=0, \ i\in
%	I(f,x^0)\}. 
%	\]
	\end{enumerate}
\end{theorem}
\begin{proof}
	Let $\mathscr{U}_2$ be the open and dense semi-algebraic set in
	$\RR^{r\times n(d)}$ in Proposition \ref{prop::u2}.

	Fix a subset $J=\{j_1,\ldots,j_s\}\subseteq\{1,\ldots,r\}$
	with $j_1<\cdots<j_s$ and $s\le n+1$. Let
	$u^J:=(u^{(j_1)},\ldots,u^{(j_s)})\in\RR^{s\times n(d)}$ for any
	$u\in\RR^{r\times n(d)}$,
	$\lambda^J:=(\lambda_{j_1},\ldots,\lambda_{j_s})\in\RR^s$,
	$z\in\RR$ and 
    \[
		\Lambda_J:=\{\lambda^J\in\RR^s :
			\lambda_{j_1}^2+\cdots+\lambda_{j_s}^2=1\}. 
	\]
%	$z^J:=\{z_{j_{s+1}},\ldots,z_{j_{r}}\}\in\RR^{r-s}$ corresponding
%	to the set $\{1,\ldots,r\}\setminus J=\{j_{s+1},\ldots,j_r\}$ with
%	$j_{s+1}<\cdots<j_r$.  
	
	We first consider the case when $s>1$ and assume that
	$J=\{1,2,\ldots,s\}$ for notational simplicity. Define 
%	\[
%		XU_J:=\{(x,u)\in\RR^n\times\mathscr{U}_2 : 
%		f_j(x,u^{(j)})-f_1(x,u^{(1)})\neq 0,\ j=s+1,\ldots,r\},
%	\]
%	and 
%	Then, $\mathscr{U}_2 \times X_J \times \Lambda$ is a semi-algebraic manifold
%	of dimension $r\times n(d)+n+s-1$. 
	the semi-algebraic map 
	\[
		\Phi_J : \mathscr{U}_2 \times \RR^n\times \Lambda_J\times \RR\rightarrow
	\RR^n\times\RR^{s-1}\times \RR
\]
by 
	\[
		\begin{aligned}
			\Phi_J(u,x,\lambda^J,z)=&\Big(\sum_{j=1}^s\lambda_j^2\nabla_x
			f_j(x,u^{(j)})^T, \ f_2(x,u^{(2)})-f_1(x,u^{(1)}),\
			\ldots,\ f_s(x,u^{(s)})-f_1(x,u^{(1)}),\\
			&z\cdot E(x,u)-1\Big),
		\end{aligned}
	\]
	where 
\[
	E(x,u):=\left\{
		\begin{aligned}
			&\prod_{j=s+1}^r(f_j(x,u^{(j)})-f_1(x,u^{(1)}))	, &
			\text{if}\ s<r,\\
			& 1 & \text{if}\ s=r.
		\end{aligned}
		\right.
\]
	Note that $\mathscr{U}_2 \times \RR^n\times \Lambda_J\times
	\RR$ is a semi-algebraic manifold of dimension $r\times
	n(d)+n+s$. Let $e_j\in\RR^{s-1}$ be the
	column vector with the $j$-th entry being $1$ and the others being
	$0$, $j=1,\ldots,s-1$, and $e=\sum_{j=1}^{s-1}e_j$. 
	A direct computation shows that
	\begin{equation}\label{eq::jac}
		\begin{aligned}
			&\left(
			\begin{array}{c|c|c|c|c|c}
				\frac{\partial \Phi_J}{\partial
					u^{(1)}_{\alpha}}&\frac{\partial \Phi_J}{\partial
						u^{(2)}_{\alpha}}&\cdots&\frac{\partial
						\Phi_J}{\partial u^{(s)}_{\alpha}}
						&\frac{\partial \Phi_J}{\partial x_i}&\frac{\partial
						\Phi_J}{\partial z}
					\end{array}
					\right)_{|\alpha|=1,\ i=1,\ldots,n}\\
					=&\left(
				\begin{array}{cccccc}
					\lambda_1^2I_n&\lambda_2^2I_n&\cdots&\lambda_s^2I_n&*&\mathbf{0}\\
					-e x^T&e_1 x^T&\cdots&e_{s-1}x^T&D&\mathbf{0}\\
					*&*&\cdots&*&*&E(x,u)
				\end{array}
				\right),
				\end{aligned}
			\end{equation}
	where $I_n$ denotes the identity matrix of order $n$, 
	\[
		D=\left(
						\begin{array}{c}
							\nabla_x f_2(x,u^{(2)})^T-\nabla_x
							f_1(x,u^{(1)})^T\\
							\vdots\\
							\nabla_x f_{s}(x,u^{(s)})^T-\nabla_x
							f_1(x,u^{(1)})^T\\
						\end{array}
						\right)\in\RR^{(s-1)\times n}.
					\]
%					and
%					\[
%						E=\prod_{j=s+1}^r
%						(f_j(x,u^{(j)})-f_1(x,u^{(1)})).
%
%						E=\left(
%						\begin{array}{ccc}
%							f_{s+1}(x,u^{(s+1)})-f_1(x,u^{(1)})&\cdots&0\\
%							\vdots&\ddots&\vdots\\
%							0&\cdots&f_r(x,u^{(r)})-f_1(x,u^{(1)})
%						\end{array}\right)
%						\in\RR^{(r-s)\times (r-s)}.
%	\]

	Now we show that $0$ is a regular value of $\Phi_J$. If
	$\Phi_J^{-1}(0)=\emptyset$, we are done; otherwise,
	fix a point $(u,x,\lambda^J,z)\in\Phi_J^{-1}(0)$, then 
	$\sum_{j=1}^s\lambda_j^2=1$, $E(x,u)\neq 0$ and $\rank D=s-1$ by
	Proposition \ref{prop::u2}. By some linear operations on the columns of
	the matrix in \eqref{eq::jac}, we obtain
\[
\left(
\begin{array}{cccccc}
					I_n&\lambda_2^2I_n&\cdots&\lambda_s^2I_n&*&\mathbf{0}\\
					0&e_1 x^T&\cdots&e_{s-1}x^T&D&\mathbf{0}\\
					\mathbf{0}&\mathbf{0}&\cdots&\mathbf{0}&\mathbf{0}&E(x,u)
				\end{array} \right),
\]
%Rearranging the columns of the above matrix yields
%\[
%\left(
%\begin{array}{cccccc}
%					I_n&*&\mathbf{0}&\lambda_2^2I_n&\cdots&\lambda_s^2I_n\\
%					0&D&\mathbf{0}&e_1 x^T&\cdots&e_{s-1}x^T\\
%					\mathbf{0}&\mathbf{0}&E&\mathbf{0}&\cdots&\mathbf{0}
%				\end{array} \right),
%\]
which implies that the rank of the matrix in \eqref{eq::jac} is $n+s$
for any
$(u,x,\lambda^J,z)\in\Phi_J^{-1}(0)$.  Hence, $0$ is a regular value of
$\Phi_J$. By the Sard's theorem with parameter (Theorem \ref{SardTheorem}), 
there exists an open and
dense semi-algebraic subset $\Sigma_J$ of $\mathscr{U}_2$ such
that for each
$u\in\Sigma_J$, $0$ is a regular value of the map
\[
	\Phi_{J,u} : \RR^n\times \Lambda_J\times \RR\rightarrow
	\RR^n\times\RR^{s-1}\times\RR,\quad (x,\lambda^J,z) \mapsto
	\Phi_J(u,x,\lambda^J,z).
\]
Since $\dim(\RR^n\times \Lambda_J\times \RR)=n+s$,  
$\Phi_{J,u}^{-1}(0)$ is either empty or a finite subset of
$\RR^n\times \Lambda_J\times \RR$. Clearly,
$\pi_J(\Phi_{J,u}^{-1}(0))$ is either empty or a finite subset of
$\RR^n$, where $\pi_J : \RR^n\times \Lambda_J\times \RR\rightarrow
\RR^n$ is the projection on the first $n$ coordinates. 

Consider the case when $s=1$ and assume $J=\{1\}$ for notational
simplicity. We can modify the the semi-algebraic map $\Phi_J$ as
	\[
		\Phi_J : \mathscr{U}_2 \times \RR^n\times \Lambda_J\times \RR\rightarrow
	\RR^n\times \RR
\]
where
	\[
		\Phi_J(u,x,\lambda^J,z)=\left(\lambda_{1}^2\nabla_x
		f_{1}(x,u^{(1)})^T,\ z\cdot
		\prod_{j=2}^r(f_j(x,u^{(j)})-f_{1}(x,u^{(1)}))-1\right).
	\]
It is straightforward to verify that analogous arguments as above
still hold, i.e., there exists an open and dense semi-algebraic subset
$\Sigma_J$ of $\mathscr{U}_2$ such that for each $u\in\Sigma_J$,
$\pi_J(\Phi_{J,u}^{-1}(0))$ is either empty or a finite subset of
$\RR^n$.

Let $\mathscr{U}_3:=\bigcap_J \Sigma_J\subseteq \mathscr{U}_2$, where
the intersection is taken over all nonempty subsets $J$ of $\{1,\ldots,r\}$ with
$\#J\le n+1$. Obviously, $\mathscr{U}_3$ is an open and dense
semi-algebraic set in $\RR^{r\times n(d)}$. 

Since $\mathscr{U}_3\subseteq\mathscr{U}_2$, 
for each $u\in\mathscr{U}_3$, $f\in\mc(F(x,u))$ and $x^0\in\cri(f)$,
we have $\# I(f,x^0)\le n+1$ by Proposition
\ref{prop::u2}. As proved in Proposition \ref{prop::cri}, it holds
that
\begin{equation}\label{eq::cri}
	\cri(\mc(F(x,u)))=\bigcup_{{J\subset\{1,\ldots,r\}}\atop{1\le \#
	J\le n+1}} \pi_J(\Phi_{J,u}^{-1}(0)). 
\end{equation}
Hence, the number of the points in $\cri(\mc(F(x,u)))$ is
			finite and bounded from above by $B_0(n,d,r)$ in
			\eqref{eq::B0}.
%by Theorem \ref{th::num}, the set $\cri(\mc(F(x,u)))$ has at
%most $\wt{B}_0(n,d,r)$ points for each $u\in\mathscr{U}_3$. 

%(ii) By the implicit function theorem, all (local) solutions
%$(x,\lambda^J,z)\in\Phi_{J,u}^{-1}(0)$ depend analytically on the
%parameter $u\in\Sigma_J$. Hence, the set-valued
%map $u\mapsto \pi_J(\Phi_{J,u}^{-1}(0))$ is continuous on the set
%$\Sigma_J$. Then, by \eqref{eq::cri}, the
%set-valued map $u\mapsto \cri(\mc(F(x,u)))$ is continuous on the set
%$\mathscr{U}_3$.

Fix $u\in\mathscr{U}_3$, $f\in\mc(F(x,u))$ and $x^0\in\cri(f)$. We
prove (i) and (ii) in the following.

(i) Fix any subset $J=\{j_1,\ldots,j_s\}\subseteq\{1,\ldots,r\}$ with
$2\le s\le n+1$ and $j_1<\cdots<j_s$. For any  
$(x,\lambda^J,z)\in\Phi_{J,u}^{-1}(0)$, we show that each $\lambda_{j_k}\neq 0,
k=1,\ldots,s$. For notational simplicity, we assume that
$J=\{1,\ldots,s\}$. Consider the Jacobian matrix of $\Phi_{J,u}$
\[
	\begin{aligned}
		D_{(x,\lambda^J,z)}\Phi_{J,u}:=&\left(
		\begin{array}{c|c|c}
			\frac{\partial \Phi_{J,u}}{\partial
			x_i}&\frac{\partial \Phi_{J,u}}{\partial
			\lambda_j}&\frac{\partial
				\Phi_{J,u}}{\partial z}
			\end{array}
			\right)_{i=1,\ldots,n,\ j=1,\ldots,s}\\
			=& \left(
				\begin{array}{ccccc}
					\sum_{j=1}^s\lambda_j^2\nabla_x^2
					f_j(x,u^{(j)})&2\lambda_1\nabla_x
					f_1(x,u^{(1)})&\cdots&2\lambda_s\nabla_x
					f_s(x,u^{(s)})& \mathbf{0}\\
					D & \mathbf{0}& \cdots& \mathbf{0}& \mathbf{0}\\
					*&\mathbf{0}&\cdots&\mathbf{0}&E(x,u)
				\end{array}
				\right)
		\end{aligned}
\]
To the contrary, suppose that $\lambda_l=0$ for some $l \in J$. 
Then for any $(x,\lambda^J,z)\in\Phi_{J,u}^{-1}(0)$, as
$\sum_{j=1}^s\lambda_j^2\nabla_x f_j(x,u^{(j)})=0$, the submatrix 
\begin{equation}\label{eq::submatrix}
\left(2\lambda_1\nabla_x f_1(x,u^{(1)})\quad \cdots\quad 2\lambda_s\nabla_x
f_s(x,u^{(s)})\right)
\end{equation}
has the rank at most $s-2$. Hence, the Jacobian matrix
$D_{(x,\lambda^J,z)}\Phi_{J,u}$ has the rank at most $n+s-1$ at any
$(x,\lambda^J,z)\in\Phi_{J,u}^{-1}(0)$. 
%Then, the $(n+l)$-th column of the Jacobian matrix of
%$\Phi_{J,u}$
%\[
%	D_{(x,\lambda^J,z)}\Phi_{J,u}:=\left(
%			\begin{array}{c|c|c}
%				\frac{\partial \Phi_{J,u}}{\partial
%				x_i}&\frac{\partial \Phi_{J,u}}{\partial
%					\lambda_j}&\frac{\partial
%						\Phi_{J,u}}{\partial z}
%					\end{array}
%					\right)_{i=1,\ldots,n,\ j=1,\ldots,s}
%\]
%
%at $(x,\lambda^J,z)$ will vanish. 
It contradicts the fact that for
any $u\in\mathcal{U}_3$, $0$ is the regular value of the map
$\Phi_{J,u}$. Denote the projection $\td{\pi}_J : \RR^n\times
\Lambda_J\times \RR^{r-s}\rightarrow \Lambda_J$.
 Let
$(\mu_j, j\in I(f,x^0))$ be any tuple satisfying \eqref{eq::lambda}.
Then, we have $(\sqrt{\mu_j}, j\in
I(f,x^0))\in\td{\pi}_{J'}(\Phi_{J',u}^{-1}(0))$ where $J'=I(f,x^0)$.
Consequently, the strict complementarity holds for $x^0$.

(ii) For notational simplicity, we assume that
$J'=I(f,x^0)=\{1,\ldots,s\}$. Then, there exists $z\in\RR$ such that
$(x^0,\sqrt{\mu_1},\ldots,\sqrt{\mu_s},z)\in\Phi_{J',u}^{-1}(0)$
at which the rank of the Jacobian matrix
$D_{(x,\lambda^{J'},z)}\Phi_{J',u}$ is $n+s$. 
For the case $s>1$, by \eqref{eq::lambda}, it is not difficult to
verify that
\[
	\begin{aligned}
		&\{y\in\RR^n : \nabla_x f_j(x^0, u^{(j)})^Ty=0,\
	j=1,\ldots,s\}\\
	&=\{y\in\RR^n : (\nabla_x f_j(x^0, u^{(j)})-\nabla_x
	f_1(x^0, u^{(1)}))^Ty=0,\ j=2,\ldots,s\}. 
\end{aligned}
\]
%In both cases of $s=1$ and $s>1$, 
To the contrary, suppose that the system $\left(\sum_{i=1}^s\mu_i
\nabla_x^2 f_i(x^0,u^{(i)})\right)y=0$ has a nonzero solution $y$ in
the set \eqref{eq::s1}.
%, then the first $n$ columns of
%$D_{(x,\lambda^{J'},z)}\Phi_{J',u}$ at
%$(x^0,\sqrt{\mu_1},\ldots,\sqrt{\mu_s},z)$ are linearly dependent. 
As the submatrix $\eqref{eq::submatrix}$ has the rank at most $s-1$ at
$(x^0,\sqrt{\mu_1},\ldots,\sqrt{\mu_s},z)$,  it is easy to see that
the rank of $D_{(x,\lambda^{J'},z)}\Phi_{J',u}$ is at most $n+s-1$ at
$(x^0,\sqrt{\mu_1},\ldots,\sqrt{\mu_s},z)$, a contradiciton.  
We omit the similar arguments for the case $s=1$. 
\end{proof}

For continuous functions $f_1,\ldots,f_r : \RR^n
	\rightarrow \RR$ and $f\in\mc(f_1,\ldots,f_r)$, we call a critical point
	$x^0\in\cri(f)$ satisfying the conditions in Proposition
	\ref{prop::u2} and Theorem \ref{th::u3}
(ii) a {\itshape nondegenerate} critical point of $f$ (c.f.,
\cite{JP1988}). By Proposition
	\ref{prop::u2} and Theorem \ref{th::u3}, all critical points of a
	CSP function selected from polynomials with generic coefficients
	are nondegenerate.

%\begin{remark}{\rm 
%If $f_l,\ldots,f_r$	are $k$-times continuously differentiable
%functions, it is
%shown in \cite[Theorem 3.1]{JP1988} that the set of
%$(f_1,\ldots,f_r)\subset C^k(\RR^n,\RR)^r$ for which every function in 
%$\mc(f_1,\ldots,f_r)$ only has nondegenerate critical points
%is $C_S^2$-open and dense in $C^k(\RR^n,\RR)^s$ in the strong
%$C_S^2$-topology. }
%\end{remark}

The following theorem shows that the critical values of all CSP
functions selected from a given set of finitely many polynomials 
with generic coefficients are distinct from each other.

\begin{theorem}\label{th::u4}
	There exists an open and dense semi-algebraic set
	$\mathscr{U}_4$ in $\RR^{r\times n(d)}$ such that the following
	property holds for any $u\in\mathscr{U}_4:$ for any distinct
		$x^0, \td{x}^0\in\cri(\mc(F(x,u)))$ and any $($not necessarily
		distinct$)$ $f, \td{f}\in\mc(F(x,u))$ with $x^0\in\cri(f)$ and
		$\td{x}^0\in\cri(\td{f}),$ it holds that $f(x^0)\neq
		\td{f}(\td{x}^0)$.
\end{theorem}
\begin{proof}
	Let $\mathscr{U}_3$ be the open and dense semi-algebraic set in
$\RR^{r\times n(d)}$ in Theorem \ref{th::u3}. 
Fix an $i\in\{1,\ldots,r\}$ and define the set 
\begin{equation}\label{eq::Lambda}
	\Lambda^{(i)}:=\left\{(\lambda_1,\ldots,\lambda_{i-1},\lambda_{i+1},\ldots,\lambda_r)
	\in\RR^{r-1} : 1-\sum_{k\neq i}\lambda_k^2>0\right\}.
\end{equation}
Define the polynomial
function $L_i : \mathscr{U}_3\times\RR^n\times\Lambda^{(i)}
\rightarrow \RR$ by 
\begin{equation}\label{eq::L_i}
	L_i(u,x,\lambda):=\sum_{k\neq i}
	\lambda_k^2(f_k(x,u^{(k)})-f_i(x,u^{(i)}))+f_i(x,u^{(i)}). 
\end{equation}
We have 
\[
	\begin{aligned}
		\left(\frac{\partial{L_i}}{\partial x_1},\ldots,
		\frac{\partial{L_i}}{\partial x_n}\right)&=\sum_{k\neq
		i}\lambda_k^2
		\left(\nabla_x f_k(x,u^{(k)})-\nabla_x f_i(x,u^{(i)})\right)^T+\nabla_x
		f_i(x,u^{(i)})^T\\
		&= \sum_{k\neq i}\lambda_k^2 \nabla_x f_k(x,u^{(k)})^T+\left(1-\sum_{k\neq
		i}\lambda_k^2\right)\nabla_x f_i(x,u^{(i)})^T
	\end{aligned}
\]
and 
\[
\frac{\partial{L_i}}{\partial \lambda_k}=2\lambda_k
(f_k(x,u^{(k)})-f_i(x,u^{(i)})),\quad k=1,\ldots,i-1,i+1,\ldots,r.
\]
Define the semi-algebraic map $\Phi_i:
\mathscr{U}_3\times\RR^n\times\Lambda^{(i)} \rightarrow \RR^{n+r-1}$
by 
\begin{equation}\label{eq::Phi_i}
	\Phi_i(u,x,\lambda):=\left(\frac{\partial{L_i}}{\partial x_1},\ldots,
		\frac{\partial{L_i}}{\partial x_n},
		\frac{\partial{L_i}}{\partial \lambda_1},
		\ldots,\frac{\partial{L_i}}{\partial \lambda_{i-1}},
		\frac{\partial{L_i}}{\partial
			\lambda_{i+1}},\ldots,\frac{\partial{L_i}}{\partial
			\lambda_r}\right). 
	\end{equation}

	It is easy to check that
	\begin{equation}\label{eq::jac2}
		\begin{aligned}
			&\left(
			\begin{array}{c|c|c|c|c|c|c}
				\frac{\partial \Phi_i}{\partial
					u^{(1)}_{\alpha}}&\cdots&\frac{\partial \Phi_i}{\partial
						u^{(i)}_{\alpha}}&\cdots&\frac{\partial
						\Phi_i}{\partial u^{(r)}_{\alpha}}
						&\frac{\partial \Phi_i}{\partial x_j}&\frac{\partial
						\Phi_i}{\partial \lambda_k}
					\end{array}
					\right)_{|\alpha|=1,\ j=1,\ldots,n,\
					k=1,\ldots,i-1,i+1,\ldots,r}\\
					=&\left(
				\begin{array}{ccccccc}
					\lambda_1^2I_n&\cdots&\left(1-\sum_{k\neq
					i}\lambda_k^2\right)I_n&\cdots&\lambda_r^2I_n&*&*\\
					2\lambda_1e_1 x^T&\cdots&-\left(\sum_{k\neq
					i}2\lambda_ke_k\right)x^T&\cdots&2\lambda_re_rx^T&D&E
				\end{array}
				\right),
				\end{aligned}
			\end{equation}
where $e_k\in\RR^{r-1}$ is the column vector with the $k$-th (resp.,
$(k-1)$-th) entry being $1$ and the others being $0$ for
$k=1,\ldots,i-1$ (resp., $k=i+1,\ldots,r$), $I_n$ is the identity matrix
of order $n$, 
	\[
		D=\left(
						\begin{array}{c}
							2\lambda_1\left(\nabla_x f_1(x,u^{(1)})^T-\nabla_x
							f_i(x,u^{(i)})^T\right)\\
							\vdots\\
							2\lambda_{i-1}\left(\nabla_x f_{i-1}(x,u^{(i-1)})^T-\nabla_x
							f_i(x,u^{(i)})^T\right)\\
							2\lambda_{i+1}\left(\nabla_x f_{i+1}(x,u^{(i+1)})^T-\nabla_x
							f_i(x,u^{(i)})^T\right)\\
							\vdots\\
							2\lambda_r\left(\nabla_x f_{r}(x,u^{(r)})^T-\nabla_x
							f_1(x,u^{(1)})^T\right)\\
						\end{array}
						\right)\in\RR^{(r-1)\times n},
					\]
					and
					\[
						E=\diag\left(
						\begin{array}{c}
						2(f_k(x,u^{(k)})-f_i(x,u^{(i)}))\\
						\vdots\\
						2(f_{i-1}(x,u^{(i-1)})-f_i(x,u^{(i)}))\\
                        2(f_{i+1}(x,u^{(i+1)})-f_i(x,u^{(i)}))\\
						\vdots\\
						2(f_r(x,u^{(r)})-f_i(x,u^{(i)}))
					\end{array}
						\right)\in\RR^{(r-1)\times (r-1)}.
	\]
By some linear operations on the columns of the matrix in
\eqref{eq::jac2}, we obtain
\[	
				\left(
			\begin{array}{ccccccc}
				\lambda_1^2I_n&\cdots&\left(1-\sum_{k\neq
				i}\lambda_k^2\right)I_n&\cdots&I_n&*&*\\
				2\lambda_1e_1 x^T&\cdots&-\left(\sum_{k\neq
				i}2\lambda_ke_k\right)x^T&\cdots&\mathbf{0}&D&E
			\end{array}
			\right),
\]

Now we show that $0$ is a regular value of $\Phi_i$. If
$\Phi_i^{-1}(0)=\emptyset$, we are done; otherwise, fix a point
$(u,x,\lambda)\in\Phi_i^{-1}(0)$ and a 
function $f\in\mc(F(x,u))$ such that $f(x)=f_i(x)$. It is obvious
that $x$ is a critical point of $f$. As $u\in\mathscr{U}_3$, by
Theorem \ref{th::u3} (i), $\lambda_k\neq 0$ for each $k$ with
$f_k(x,u^{(k)})-f_i(x,u^{(i)})=0$. As
$\mathscr{U}_3\subseteq\mathscr{U}_2$ where $\mathscr{U}_2$ is the open and
dense semi-algebraic subset in $\RR^{r\times n(d)}$ in Proposition
\ref{prop::u2}, the rank of the matrix in \eqref{eq::jac2}
%$D_{(u,x,\lambda)}\Phi_i$ at $(u,x,\lambda)$ 
is $n+r-1$ by Proposition
\ref{prop::u2}. Hence, $0$ is a regular value of $\Phi_i$. 
By the Sard's theorem with parameter (Theorem \ref{SardTheorem}), there
exists an open and
dense semi-algebraic subset $\mathscr{U}^{(i)}$ of $\mathscr{U}_3$ such
that for each $u\in\mathscr{U}^{(i)}$, $0$ is a regular value of the map
\[
	\Phi_{i,u} : \RR^n\times \Lambda^{(i)}\rightarrow
	\RR^{n+r-1},\quad (x,\lambda) \mapsto \Phi_{i,u}(u,x,\lambda).
\]

For any $i,j\in\{1,\ldots,r\}$, define the semi-algebraic map
\[
	\Psi_{i,j} :
	(\mathscr{U}^{(i)}\cap\mathscr{U}^{(j)})\times\left(\left((\RR^n\times\Lambda^{(i)})
	\times(\RR^n\times\Lambda^{(j)})\right)\setminus\mathcal{E}\right)
	\rightarrow \RR\times\RR^{n+r-1}\times\RR^{n+r-1}
\]
by 
\[
	\Psi_{i,j}(u,x,\lambda,\td{x},\td{\lambda})=\left(L_i(u,x,\lambda)-L_j(u,\td{x},\td{\lambda}),\ 
	\Phi_i(u,x,\lambda),\ \Phi_j(u,\td{x},\td{\lambda})\right),
\]
where
\[
	\mathcal{E}:=\{(x,\lambda,\td{x},\td{\lambda})\in(\RR^n\times\Lambda^{(i)})
	\times(\RR^n\times\Lambda^{(j)}) : x=\td{x}\}.
\]
It is easy to check that
\begin{equation}\label{eq::jac3}
	\begin{aligned}
		& \left(
		\begin{array}{c|c|c}
			\frac{\partial{\Psi_{i,j}}}{\partial{u^{(k)}_{\alpha}}}
			&D_{(x,\lambda)}\Psi_{i,j}&D_{(\td{x},\td{\lambda})}\Psi_{i,j}
		\end{array}
		\right)_{k=1,\ldots,r,\ |\alpha|=1}\\
	&=\left(
		\begin{array}{ccc}
			A&\Phi_i(u,x,\lambda)&\Phi_j(u,\td{x},\td{\lambda})\\
			*&D_{(x,\lambda)}\Phi_i(u,x,\lambda)&\mathbf{0}\\
			*&\mathbf{0}& D_{(\td{x},\td{\lambda})}\Phi_j(u,\td{x},\td{\lambda})
		\end{array}
		\right),
	\end{aligned}
\end{equation}
where
\[
	D_{(x,\lambda)}:=\left(\frac{\partial}{\partial x_1}, \ldots,
	\frac{\partial}{\partial x_n}, \frac{\partial}{\partial \lambda_1}, 
	\ldots, \frac{\partial}{\partial \lambda_{i-1}},
	\frac{\partial}{\partial \lambda_{i+1}},\ldots, \frac{\partial}{\partial
	\lambda_r}\right),
\]
\[
	D_{(\td{x},\td{\lambda})}:=\left(\frac{\partial}{\partial \td{x}_1}, \ldots,
	\frac{\partial}{\partial \td{x}_n}, \frac{\partial}{\partial
		\td{\lambda}_1}, 
		\ldots, \frac{\partial}{\partial \td{\lambda}_{j-1}},
		\frac{\partial}{\partial \td{\lambda}_{j+1}},\ldots, \frac{\partial}{\partial
			\td{\lambda}_r}\right),
\]
\[
	\begin{aligned}
		A&=\left(
		\begin{array}{c|c|c|c|c|c|c}
			\frac{\partial{(L_i-L_j)}}{\partial{u^{(1)}_{\alpha}}}&\cdots&\frac{\partial{(L_i-L_j)}}{\partial{u^{(i)}_{\alpha}}}
			&\cdots&\frac{\partial{(L_i-L_j)}}{\partial{u^{(j)}_{\alpha}}}&\cdots&\frac{\partial{(L_i-L_j)}}{\partial{u^{(r)}_{\alpha}}}
		\end{array}
		\right)_{|\alpha|=1}\\
		&=\left(
		\begin{array}{ccccccc}
			\lambda_1^2x-\td{\lambda}_1^2\td{x}&\cdots&\left(1-\sum_{k\neq
			i}\lambda_k^2\right)x-\td{\lambda}_i^2\td{x}
			&\cdots&\lambda_j^2x-\left(1-\sum_{k\neq
			j}\td{\lambda}_k^2\right)\td{x}&\cdots&\lambda_r^2x-\td{\lambda}_r^2\td{x}
		\end{array}
		\right).
	\end{aligned}
\]

Now we show that $0$ is a regular value of $\Psi_{i,j}$. If
$\Psi_{i,j}^{-1}(0)=\emptyset$, we are done; otherwise,
fix a point $(u,x,\lambda,x,\td{\lambda})\in\Psi_{i,j}^{-1}(0)$. We
have $x\neq \td{x}$, $u\in\mathscr{U}^{(i)}\cap\mathscr{U}^{(j)}$ and 
\[
	\Phi_i(u,x,\lambda)=\Phi_j(u,\td{x},\td{\lambda})=0.
\]
Since $0$ is a regular value of the maps $\Phi_{i,u}$ and
$\Phi_{j,u}$, it holds that
\[
	\rank\ D_{(x,\lambda)}\Phi_i(u,x,\lambda)=\rank\
	D_{(\td{x},\td{\lambda})}\Phi_j(u,\td{x},\td{\lambda})=n+r-1. 
\]
Note that by some linear operations on the columns of the matrix $A$, we obtain
\[
	\left(
		\begin{array}{ccccccc}
			\lambda_1^2x-\td{\lambda}_1^2\td{x}&\cdots&\left(1-\sum_{k\neq
			i}\lambda_k^2\right)x-\td{\lambda}_i^2\td{x}
			&\cdots&\lambda_j^2x-\left(1-\sum_{k\neq
			j}\td{\lambda}_k^2\right)\td{x}&\cdots&x-\td{x}
		\end{array}
		\right).
\]
Since $x\neq\td{x}$, we obtain that the rank of the matrix in
\eqref{eq::jac3} is $1+2(n+r-1)$ at
any $(u,x,\lambda,\td{x},\td{\lambda})\in\Psi_{i,j}^{-1}(0)$ and hence
$0$ is a regular value of $\Psi_{i,j}$. 
By the Sard's theorem with parameter (Theorem \ref{SardTheorem}), 
there exists an open and
dense semi-algebraic subset $\mathscr{U}^{(i,j)}$ of
$\mathscr{U}^{(i)}\cap\mathscr{U}^{(j)}$ such
that for each $u\in\mathscr{U}^{(i,j)}$, $0$ is a regular value of the map
\[
	\begin{aligned}
		\Psi_{i,j,u} : \left((\RR^n\times\Lambda^{i})
		\times(\RR^n\times\Lambda^{(j)})\right)\setminus\mathcal{E}
		&\rightarrow \RR\times\RR^{n+r-1}\times\RR^{n+r-1}\\
	(x,\lambda,\td{x},\td{\lambda})&\mapsto
	\Psi_{i,j}(u,x,\lambda,\td{x},\td{\lambda}).
	\end{aligned}
\]
Note that 
\[
	\dim(\left((\RR^n\times\Lambda^{i})
		\times(\RR^n\times\Lambda^{(j)})\right)\setminus\mathcal{E})=2(n+r-1)
		<\dim(\RR\times\RR^{n+r-1}\times\RR^{n+r-1}).
\]
Hence, $\Psi_{i,j,u}^{-1}(0)=\emptyset$.  Let 
\[
	\mathscr{U}_4=\bigcap_{i,j\in\{1,\ldots,r\}} \mathscr{U}^{(i,j)}. 
\]
Then, $\mathscr{U}_4$ is an open and dense semi-algebraic set in
$\RR^{r\times n(d)}$. 

For any $u\in\mathscr{U}_4$, fix two distinct $x^0,
\td{x}^0\in\cri(\mc(F(x,u)))$ and any $f, \td{f}\in\mc(F(x,u))$ such
that $x^0\in\cri(f)$ and $\td{x}^0\in\cri(\td{f})$.  Fix two indices
$i\in I(f,x^0)$ and $j\in I(\td{f},\td{x}^0)$. As
$\mathscr{U}_4\subseteq\mathscr{U}_3$ and $x^0\in\cri(f)$, by Theorem
\ref{th::u3} (i), 
%and $\td{x}^0\in\cri(\td{f})$, 
there exists
$\lambda=(\lambda_1,\ldots,\lambda_{i-1},\lambda_{i+1},\ldots,\lambda_r)\in\Lambda^{(i)}$
such that 
\[
	\lambda_k=0,\ \forall k\not\in I(f,x^0),\quad\sum_{k\neq
	i}\lambda_k^2\left(\nabla_x f_k(x^0,u^{(k)})-\nabla_x
	f_i(x^0,u^{(i)})\right)+\nabla_x f_i(x^0,u^{(i)})=0.
\]
Similarly, there exists 
$\td{\lambda}=(\td{\lambda}_1,\ldots,\td{\lambda}_{j-1},\td{\lambda}_{j+1},\ldots,\td{\lambda}_r)\in\Lambda^{(j)}$
such that 
\[
	\td{\lambda}_k=0,\ \forall k\not\in I(\td{f},\td{x}^0),\quad\sum_{k\neq
	j}\td{\lambda}_k^2\left(\nabla_x f_k(\td{x}^0,u^{(k)})-\nabla_x
	f_j(\td{x}^0,u^{(j)})\right)+\nabla_x f_j(\td{x},u^{(j)})=0.
\]
Therefore, we obtain 
\[
	\begin{aligned}
		L_i(u,x^0,\lambda)&=f_i(x^0,u^{(i)})=f(x^0),\\
		L_j(u,\td{x}^0,\td{\lambda})&=f_j(\td{x}^0,u^{(j)})=\td{f}(\td{x}^0),\\
		\Phi_i(u,x^0,\lambda)&=\Phi_j(u,\td{x}^0,\td{\lambda})=0. 
	\end{aligned}
\]
By the definition of $\mathscr{U}_4$,
$\Psi_{i,j,u}^{-1}(0)=\emptyset$. Hence, we must have $f(x^0)\neq
\td{f}(\td{x}^0)$. The conclusion follows. 
\end{proof}

Consequently, the uniqueness of optimal solutions for global optimization problems
with CSP functions is a generic property.  

\begin{corollary}\label{cor::u4}
	Let $\mathscr{U}_4$ be open and dense semi-algebraic set the in
	$\RR^{r\times n(d)}$ in Theorem \ref{th::u4}. Then, for any
	$u\in\mathscr{U}_4$ and $f\in\mc(F(x,u)),$ the optimization
	problem $\min_{x\in\RR^n} f(x)$ has at most one optimal solution.
\end{corollary}

The following result shows that the goodness at infinity is a
generic property of CSP functions.
Recall the non-smooth slope $\mathfrak{m}^{\circ}_f(x)$ in Definition
\ref{def::slope}.

\begin{theorem}\label{th::u5}
	There exists an open and dense semi-algebraic set
	$\mathscr{U}_5$ in $\RR^{r\times n(d)}$ such that for any
	$u\in\mathscr{U}_5$ and $f\in\mc(F(x,u)),$ $f$ is ``good at
	infinity'' in the sense that there exist constants $c>0$ and $R>0$
	such that $\mathfrak{m}^{\circ}_f(x)\ge c$ for any $x\in\RR^n$
	with $\Vert x\Vert\ge R$. 
\end{theorem}
\begin{proof}
Let $\mathscr{U}_3$ be the open and dense semi-algebraic set in
$\RR^{r\times n(d)}$ in Theorem \ref{th::u3}. Recall the set
$\Lambda^{(i)}$ in \eqref{eq::Lambda} and the polynomial function
$L_i$ in \eqref{eq::L_i}. For $i=1,\ldots,r$, define 
\begin{equation}\label{eq::LL_i}
	\mathscr{L}_i:=\left\{u\in\mathscr{U}_3 : 
	\left\{
	\begin{aligned}
		&\exists
		\{(x^{(l)},\lambda^{(l)})\}_{l\in\N}\subset\RR^n\times\Lambda^{(i)}\
		\text{such that}\\
		&\lim_{l\rightarrow\infty} \Vert
		x^{(l)}\Vert=+\infty,\ \frac{\partial L_i}{\partial
		\lambda_k}(u,x^{(l)},\lambda^{(l)})=0,\ k\neq i,\\
		&\lim_{l\rightarrow\infty} \frac{\partial L_i}{\partial
		x_j}(u,x^{(l)},\lambda^{(l)})=0,\ j=1,\ldots,n,
	\end{aligned}\right.
\right\}.
\end{equation}
We will show that $\mathscr{L}_i$ is a semi-algebraic set of dimension
at most $r\times n(d)-1$.

Recall the semi-algebraic map $\Phi_i$ in \eqref{eq::Phi_i}. Let
$\mathcal{C}_i$ be the closure of the set
$\Phi_i^{-1}(0)\subset\mathscr{U}_3\times\RR^n\times\Lambda^{(i)}$ in
$\RR^{r\times n(d)}\times\PP^n\times\RR^{r-1}$ where $\PP^n$ is the
real projective space. As the sets $\mathcal{C}_i$ and
$\Phi_i^{-1}(0)$ are semi-algebraic, so is the set
$\mathcal{C}_i\setminus\Phi_i^{-1}(0)$.

Let $\pi: \RR^{r\times
n(d)}\times\PP^n\times\RR^{r-1}\rightarrow \RR^{r\times n(d)}$ be the
projection on the first $r\times n(d)$ coordinates. Assume that
$\mathscr{L}_i\neq\emptyset$ and fix a point
$\bar{u}\in\mathscr{L}_i$. We will show that
$\bar{u}\in\pi(\mathcal{C}_i\setminus\Phi_i^{-1}(0))$. 
For $\bar{u}$, there exists a sequence
$\{(x^{(l)},\lambda^{(l)})\}_{l \in \N}\subset\RR^n\times\Lambda^{(i)}$ satisfying
the conditions in \eqref{eq::LL_i}.
For each $j=1,\ldots,n$, we have 
\[
	\frac{\partial L_i}{\partial x_j}(u,x,\lambda)=\left[\sum_{k\neq
	i}\lambda_k^2 u_{e_j}^{(k)}+\left(1-\sum_{k\neq
	i}\lambda_k^2\right)u_{e_j}^{(i)}\right]+A_{i,j},
\]
where $e_j\in\RR^n$ is the vector whose $j$-th entry is $1$ and the
others are $0$ and $A_{i,j}$ is a polynomial in $(u^{(k)}_\alpha, x,
\lambda)$ with $k=1,\ldots,r$ and $\vert\alpha\vert>1$. Let
\[
	y_j^{(l)}=\frac{\partial L_i}{\partial x_j}(\bar{u}, x^{(l)},\lambda^{(l)}),\
j=1,\ldots,n,
\]
and $u^{(l)}=(u^{(l,1)},\ldots,u^{(l,r)})\in\RR^{r\times n(d)}$ where
$u^{(l,k)}=(u^{(l,k)}_\alpha)_{|\alpha|\le d}\in\RR^{n(d)}$ such that
for $k=1,\ldots,r$, 
\[
	u^{(l,k)}_{e_j}=\bar{u}^{(k)}_{e_j}-y_j^{(l)},\ j=1,\ldots,n,\quad
	u^{(l,k)}_{\alpha}=\bar{u}^{(k)}_\alpha,\ \vert\alpha\vert\neq 1. 
\]
Clearly, $\lim_{l\rightarrow\infty} u^{(l)}=\bar{u}$. It is easy to
check that for each $l\in\N$, 
\[
\frac{\partial L_i}{\partial
\lambda_k}(u^{(l)},x^{(l)},\lambda^{(l)})=0,\ k\neq i,\quad
\frac{\partial L_i}{\partial x_j}(u^{(l)},x^{(l)},\lambda^{(l)})=0,\ j=1,\ldots,n.
\]
That is, $(u^{(l)},x^{(l)},\lambda^{(l)})\in\Phi_i^{-1}(0)$ for all
$l\in\N$, which implies that
$\bar{u}\in\pi(\mathcal{C}_i\setminus\Phi_i^{-1}(0))$. 

Recall that $0$ is a regular value of $\Phi_i$ as proved in Theorem
\ref{th::u4} and hence 
$\dim(\Phi_i^{-1}(0))=r\times n(d)$. It follows from Proposition
\ref{DimensionProposition} that
\[
	\dim(\mathscr{L}_i)\le
	\dim(\pi(\mathcal{C}_i\setminus\Phi_i^{-1}(0)))\le
	\dim(\mathcal{C}_i\setminus\Phi_i^{-1}(0))<\dim(\Phi_i^{-1}(0))=r\times
	n(d).
\]
Let $\mathscr{U}_5:=\mathscr{U}_3\setminus \overline{\bigcup_{1\le
i\le r}\mathscr{L}_i}$. Then, $\mathscr{U}_5$ is an open and dense
semi-algebraic set in $\RR^{r\times n(d)}$. 

Fix a point $u'\in\mathscr{U}_5$ and a function $f\in\mc(F(x,u'))$. If
$f$ is not good at infinity, then there exist $(x^{(l)})_{l\in\N}$ and
$(\mu_i^{(l)}\in\RR, i\in I(f,x^{(l)}))_{l\in\N}$ such that
$\lim_{l\rightarrow\infty}\Vert x^{(l)}\Vert=+\infty$, 
\[
	\mu_i^{(l)}\ge 0, \sum_{i\in
		I(f,x^{(l)})}\mu_i^{(l)}=1\quad\text{and}\quad
		\lim_{l\rightarrow\infty}\Big\Vert \sum_{i\in
			I(f,x^{(l)})}\mu_i^{(l)}\nabla f_i(x^{(l)})\Big\Vert=0.
\]
%$\lim_{l\rightarrow\infty} \mathfrak{m}^{\circ}_f(x^{(l)})=0$. 
By passing to a subsequence if necessary,
we may assume that there exists an index $i^0\in\{1,\ldots,r\}$ such
that $i^0\in I(f,x^{(l)})$ and $\mu^{(l)}_{i^0}>0$ for all $l\in\N$. It
is obvious that
$u'\in\mathscr{L}_{i^0}$, a contradiction. Therefore, there exist
constants $c_f>0$ and $R_f>0$ such that $\mathfrak{m}^{\circ}_f(x)\ge c_f$ for
any $x\in\RR^n$ with $\Vert x\Vert\ge R_f$. Let $c:=\min\{c_f :
f\in\mc(F(x,u'))\}$ and $R:=\max\{R_f : f\in\mc(F(x,u'))\}$. 
Since $\mc(F(x,u'))$ has only finitely many functions by Theorem
\ref{th::Lip}, $c$ and $R$ are positive. 
Clearly, the conclusion holds for $c$ and $R$.
\end{proof}

By means of Theorem \ref{th::u5}, we can establish the coercivity of
CSP functions selected from polynomials with generic coefficients.

\begin{theorem}\label{th::u5-2}
	Let $\mathscr{U}_5$ be the open and dense semi-algebraic set in
	$\RR^{r\times n(d)}$ in Theorem \ref{th::u5}. For any
	$u\in\mathscr{U}_5$ and $f\in\mc(F(x,u)),$ if $f$ is bounded from
	below, then there exist constants $\td{c}>0$ and $\wt{R}>0$ such
	that $f(x)\ge \td{c}\Vert x\Vert$ for any $x\in\RR^n$ with $\Vert
	x\Vert\ge \wt{R}$. In particular$,$ $f$ is coercive on $\RR^n$. 
\end{theorem}
\begin{proof}
	Let $c>0$ and $R>0$ be the constants in the statement of Theorem
	\ref{th::u5}.
	Fix a point $u\in\mathscr{U}_5$	and a function $f\in\mc(F(x,u))$.
	Then, $\mathfrak{m}^{\circ}_f(x)\ge c$ for any $x\in\RR^n$ with
	$\Vert x\Vert\ge R$. Denote $f^*:=\inf_{x\in\RR^n} f(x)$ and assume
	that $f^*>-\infty$. Let 
	\[
		\td{c}:=\frac{c}{4}>0\quad\text{and}\quad
		\wt{R}:=\max\left\{2R, \frac{4|f^*|}{c}\right\}>0.
	\]
	We will show that $f(x)\ge \td{c}\Vert x\Vert$ for any $x\in\RR^n$
	with $\Vert x\Vert\ge \wt{R}$. 
	
	To the contrary, suppose that there exists a point
	$x^{0}\in\RR^n$ such that $\Vert x^{0}\Vert\ge \wt{R}$ and
	$f(x^{0})<\td{c}\Vert x^{0}\Vert$. As $\Vert x^{0}\Vert\ge
	R$, we have $\mathfrak{m}^{\circ}_f(x^{0})\ge c$. Then,
	$f(x^{0})>f^*$ since otherwise $x^{0}\in\cri(f)$ and
	$\mathfrak{m}^{\circ}_f(x^{0})=0$. Let $\varepsilon:=f(x^{0})-f^*>0$ and
	$\lambda:=\frac{\Vert x^{0}\Vert}{2}>0$. Then, by the Ekeland
	Variation Principle (Theorem \ref{th::ekeland}), there is some
	point $y^{0}\in\RR^n$ such that 
	\[
		\begin{aligned}
			&f(y^{0})\le f(x^{0}),\\
			&\Vert y^{0}-x^{0}\Vert\le\lambda,\\
			f(y^{0})\le f(x)&+\frac{\varepsilon}{\lambda}\Vert
			x-y^{0}\Vert\quad \text{for all}\ x\in\RR^n.
		\end{aligned}
	\]
	We have 
	\[
		\Vert y^{0}\Vert\ge \Vert x^{0}\Vert-\Vert
		x^{0}-y^{0}\Vert\ge \Vert
		x^{0}\Vert-\lambda=\frac{\Vert x^{0}\Vert}{2}\ge R. 
	\]
	Obviously, $y^{0}$ is a global minimizer of the function
	\[
		\RR^n \rightarrow \RR,\quad x \mapsto
		f(x)+\frac{\varepsilon}{\lambda}\Vert x-y^{0}\Vert.
	\]
	Then by Proposition \ref{prop::clarke} and Example~\ref{exam:1}, we get
	\[
		\begin{aligned}
			0\in \partial^{\circ}\left(f(\cdot)+\frac{\varepsilon}{\lambda}\Vert
			\cdot -y^{0}\Vert\right)(y^{0})&\subseteq \partial^{\circ}
			f(y^{0})+\frac{\varepsilon}{\lambda}\partial^{\circ}\Vert
			\cdot -y^{0}\Vert(y^{0}) &= \partial^{\circ}
			f(y^{0})+\frac{\varepsilon}{\lambda}\bar{\mathbb{B}}.
		\end{aligned}
	\]
	By definition,
	\[
		\begin{aligned}
			\mathfrak{m}^{\circ}_f(y^{0})&\le
			\frac{\varepsilon}{\lambda}=\frac{2(f(x^{0})-f^*)}{\Vert
				x^{0}\Vert}< \frac{2(\td{c}\Vert
					x^{0}\Vert-f^*)}{\Vert x^{0}\Vert}\\
					& \le \frac{2(\td{c}\Vert x^{0}\Vert+|f^*|)}{\Vert x^{0}\Vert}
					\le \frac{2(\td{c}\Vert
					x^{0}\Vert+\td{c}\wt{R})}{\Vert x^{0}\Vert}\le
					4\td{c}=c,
				\end{aligned}
	\]
	a contradiction.
\end{proof}
By Corollary \ref{cor::u4} and Theorem \ref{th::u5-2}, 
if a CSP function selected from polynomials with generic coefficients
is bounded from below, then its global minimum is attainable at a
unique minimizer.

\section{Discussions on non-smooth {\L}ojasiewicz's inequality and
error bound}\label{sec::errorbound}

An error bound for a set $S\subset\RR^n$ is an inequality that bounds the
distance from an arbitrary point $x$ in a test set to $S$ in terms of
the amount of ``constraint violation'' at $x$, called its {\itshape
residual}. Error bounds have numerous applications in many fields. For
example, they can be used to establish the rate of convergence of many
optimization methods. 

Some results about
error bounds with explicit exponents for the polynomial system
\[
	S:=\{x\in\RR^n : f_1(x)\le 0, \ldots, f_r(x)\le 0\}\quad \text{where}\ 
	f_1,\ldots,f_r\in\RR[x],
\]
are derived in \cite{DHP2017,LMP2015}. Precisely, they define the
non-smooth slope of a continuous function $f$ by its limiting
subdifferential $\partial f$ (compare with Definition \ref{def::slope})
\[
	\mathfrak{m}_f(x)=\inf\{\Vert v\Vert : v\in\partial f(x)\}. 
\]
Then, a non-smooth {\L}ojasiewicz's inequality about the non-smooth
slope is established for the maximum function 
\[
	f_{\max}(x):=\max\{f_1(x),\ldots,f_r(x)\}. 
\]
By invoking the Ekeland Varaitional Principle and the properties in
Proposition \ref{prop::clarke} for the limiting subdifferential (see
Remark \ref{rk::pp}), some local and global H\"olderian error
bounds with explicit exponents for $S$ are obtained. 

Note that the maximum function $f_{\max}$ belongs to the set
$\mc(f_1,\ldots,f_r)$.
According to \cite[Theorem 3.46 (ii)]{Mordukhovich2006}, the non-smooth
slopes $\mathfrak{m}_{f_{\max}}(x)$ for $f_{\max}$ defined via the limiting
subdifferential and $\mathfrak{m}^{\circ}_f(x)$ in Definition
\ref{def::slope} for any $f\in\mc(f_1,\ldots,f_r)$ defined via the Clarke
subdifferential have the same representation as in \eqref{eq::mf}.
Note also that the properties in Proposition \ref{prop::clarke} hold for
both the Clarke subdifferential and the limiting subdifferential.
Consequently, some results about
non-smooth {\L}ojasiewicz's inequality and error bounds for the
maximun function $f_{\max}$ in \cite{DHP2017,LMP2015} are also valid
for any function $f\in\mc(f_1,\ldots,f_r)$. In the following, we
specify some of them and refer the readers to \cite{DHP2017,LMP2015}
or \cite[Section 3]{HaHV2017} for the proofs and more analogous
results. 

The following non-smooth {\L}ojasiewicz's inequality holds for any CSP
functions.

\begin{theorem}\label{th::mf}
	Let $f\in\mc(f_1,\ldots,f_r)$ where $f_1,\ldots,f_r\in\RR[x]$
	with degree bounded by $d$ and $\bar{x}\in\mathbb{R}^n$ with $f(\bar{x})=0$.
Then there are numbers $c>0$ and $\varepsilon>0$ such that
\[
	\mathfrak{m}^{\circ}_{f}(x)\ge
	c|f(x)|^{1-\frac{1}{\mathscr{L}(n,d,r)}},\quad\text{for all}\
	x\in\mathbb{R}^n\ \text{with}\
	\Vert x-\bar{x}\Vert\le \varepsilon,
\]
where 
\begin{equation}\label{eq::LL}
	\mathscr{L}(n,d,r):=(d+1)(3d)^{n+r-2}. 
\end{equation}
\end{theorem}
\begin{proof}
	In view of the above discussions, it follows from the proof of
	\cite[Theorem 3.2]{HaHV2017}. 
\end{proof}

Some results about local and global H\"olderian error bounds for CSP
functions are listed below.
\begin{theorem}
Let $f\in\mc(f_1,\ldots,f_r)$ where $f_1,\ldots,f_r\in\RR[x]$
with degree bounded by $d$ and $S:=\{x\in\RR^n : f(x)\le
0\}\neq\emptyset$. Then for any
compact $K \subset \mathbb{R}^n,$ there exists a constant $c > 0$ such
that 
\begin{eqnarray*}
	\frac{c}{\mathscr{L}(n,d,r)}\ \mathrm{dist}(x, S)& \le&
	[f(x)]_+^{\frac{1}{\mathscr{L}(n,d,r)}},  \quad \textrm{ for all }
	\quad x \in K,
\end{eqnarray*}
where $[f(x)]_+:=\max\{f(x),0\}$ and $\mathrm{dist}(x, S)$ denotes the
Euclidean distance from $x$ to $S$. 
\end{theorem}
\begin{proof}
	As $K$ is compact, we only need to prove that for any
	$\bar{x}\in\RR^n$, there are constants $c(\bar{x})> 0$ and
	$\varepsilon(\bar{x})>0$ such that for all $\Vert x-\bar{x}\Vert\le
	\varepsilon(\bar{x})$,
	\[
		\frac{c(\bar{x})}{\mathscr{L}(n,d,r)}\ \mathrm{dist}(x, S) \le
		[f(x)]_+^{\frac{1}{\mathscr{L}(n,d,r)}}.
	\]
	It is easy to see that the above statement holds for any
	$\bar{x}$ with $f(\bar{x})<0$ or $f(\bar{x})>0$.
	For every $\bar{x}\in K$ with $f(\bar{x})=0$, by Theorem \ref{th::mf}, 
	there are numbers $c(\bar{x})> 0$, 
	$\varepsilon(\bar{x})>0$ which depend on $\bar{x}$ and $0<\frac{1}{\mathscr{L}(n,d,r)}<1$
defined in \eqref{eq::LL}, such that
\[
	\mathfrak{m}^{\circ}_{f}(x)\ge
	c(\bar{x})|f(x)|^{1-\frac{1}{\mathscr{L}(n,d,r)}},\quad\text{for all}\
	x\in\mathbb{R}^n\ \text{with}\
	\Vert x-\bar{x}\Vert\le \varepsilon(\bar{x}).
\]
Then, for all $x\in\mathbb{R}^n$ with $\Vert x-\bar{x}\Vert\le
\frac{\varepsilon(\bar{x})}{2}$,
by \cite[Lemma 3.3]{HaHV2017}, we have 
\[
	\frac{c(\bar{x})}{\mathscr{L}(n,d,r)}\mathrm{dist}(x, S)\le
	[f(x)]_+^{\frac{1}{\mathscr{L}(n,d,r)}}.
\]
Thus, the conclusion follows. 
%As $K$ is compact, there are finitely many
%$x^{(1)},\ldots,x^{(s)}\in K\cap [f=0]$ such that 
%\
%\[
%	(K\cap [f=0])\subset\cup_{i=1}^s B(x^{(i)},\varepsilon(x^{(i)})/2).
%\]
%Let
%\[
%	c:=\min\{c(x^{(i)})(1-\mathcal{L}(n,d,r)) : i=l,\ldots,s\}
%\]
%and $\alpha(n,d,r):=1-\mathcal{L}(n,d,r)$. 
%Then 
%\begin{eqnarray*}
%[f(x)]_+^\alpha &\ge& c \, \mathrm{dist}(x, [f \le 0]) \quad \textrm{
%for all } \quad x \in K.
%\end{eqnarray*}
\end{proof}

\begin{theorem}\label{th::eb}
Let $f\in\mc(f_1,\ldots,f_r)$ where $f_1,\ldots,f_r\in\RR[x]$
with degree bounded by $d$ and $S:=\{x\in\RR^n : f(x)\le
0\}\neq\emptyset$. 
Assume that there exist constants $c>0$ and $R>0$ such that
$\mathfrak{m}^{\circ}_f(x)\ge c$ for any $x\in\RR^n$ with $\Vert
x\Vert\ge R,$ then there exist constants $\bar{c}>0$ and $\alpha>0$
such that 
\begin{equation}\label{eq::eb}
	\bar{c}\ \mathrm{dist}(x, S)\le
	[f(x)]_+^{\alpha}+[f(x)]_+,\quad\text{for all}\ x\in\RR^n,
\end{equation}
i.e., $f$ admits a global H\"olderian error bound.
\end{theorem}
\begin{proof}
	It can be shown by replacing $\mathfrak{m}_f(x)$ by
	$\mathfrak{m}^{\circ}_f(x)$ in the proof of \cite[Theorem 3.6]{HaHV2017}. 
\end{proof}

In fact, the global H\"olderian error bound \eqref{eq::eb} is
a generic property for CSP functions. 

\begin{corollary}\label{cor::u5}
	Let $\mathscr{U}_5$ be the open and dense semi-algebraic set in
	$\RR^{r\times n(d)}$ in Theorem \ref{th::u5}. Then for any
	$u\in\mathscr{U}_5$ and $f\in\mc(F(x,u)),$ $f$ admits the global
	H\"olderian error bound \eqref{eq::eb}. 
\end{corollary}
\begin{proof}
	It follows from Theorems \ref{th::u5} and \ref{th::eb}. 
\end{proof}

Some other global H\"olderian error bounds for functions 
in $\mc(f_1,\ldots,f_r)$ where $ f_1,\ldots,f_r\in\RR[x]$ can be
derived under the {\itshape Palais-Smale condition} or the
{\itshape non-degeneracy condition} on $f_1,\ldots,f_r$. We refer the
reader to \cite[Section 3]{HaHV2017} for the details and analogous
proofs.

%\begin{property}
%Let $f \colon \mathbb{R}^n \rightarrow \mathbb{R}$ be a piecewise polynomial function of degree $d.$ Then for any compact $K \subset \mathbb{R}^n,$ there exists a constant $c > 0$ such that 
%\begin{eqnarray*}
%[f(x)]_+ &\ge& c \, \mathrm{dist}(x, [f \le 0])^\alpha \quad \textrm{ for all } \quad x \in K,
%\end{eqnarray*}
%where $\alpha = \alpha(n, d, r).$
%\end{property}

\section{Conclusions}\label{sec::con}
In this paper, we have obtained some properties satisfied by all CSP functions
selected from a given set of finitely many polynomials. In particular, we
show that there are only finitely many of such CSP functions and each
of them is semi-algebraic. Then, we derive the following generic
properties for all CSP functions
selected from the same set of finitely many polynomials: (i) the
critical points of all those CSP function are finite and the
corresponding critical values are all distinct; (ii) each of those CSP
functions is ``good at infinity''; (iii) each of those CSP functions
which is bounded from below is coercive. We have also discussed some
results about the non-smooth {\L}ojasiewicz's inequality and error
bound for CSP functions. 
The stability and genericity for CSP function
optimization problems over semi-algebraic sets will be considered in
the future work.

\subsection*{Acknowledgments}
The authors wish to thank Ti\'{\^{e}}n-S\OW n Ph\d{a}m for
kindly providing us the papers \cite{LeePham2016,LeePham2017} and many
fruitful discussions.
Feng Guo is supported by the Chinese National Natural Science
Foundation under grant 11571350, the Fundamental Research
Funds for the Central Universities. 
Liguo Jiao is supported by Jiangsu Planned Projects for Postdoctoral Research
Funds 2019 (no. 2019K151).
Do Sang Kim is supported by the National Research Foundation of Korea Grant funded by
the Korean Government (NRF-2019R1A2C1008672).

%\def\refname{\Large\bfseries References}
%\bibliographystyle{abbrv}
%\bibliography{CSP.bib}

\begin{thebibliography}{10}

\bibitem{Absil2006}
P.~A. Absil and K.~Kurdyka.
\newblock On the stable equilibrium points of gradient systems.
\newblock {\em Systems \& Control Letters}, 55(7):573 -- 577, 2006.

\bibitem{APS1997}
A.~A. Agrachev, D.~Pallaschke, and S.~Scholtes.
\newblock On {Morse} theory for piecewise smooth functions.
\newblock {\em Journal of Dynamical and Control Systems}, 3(4):449--469, 1997.

\bibitem{AHO1997}
F.~Alizadeh, J.-P.~A. Haeberly, and M.~L. Overton.
\newblock Complementarity and nondegeneracy in semidefinite programming.
\newblock {\em Mathematical Programming}, 77(1):111--128, 1997.

\bibitem{Aze}
D.~Az{\'e}.
\newblock A survey on error bounds for lower semicontinuous functions.
\newblock {\em ESAIM: Proceedings}, 13:1--17, 2003.

\bibitem{BARTELS1995385}
S.~G. Bartels, L.~Kuntz, and S.~Scholtes.
\newblock Continuous selections of linear functions and nonsmooth critical
  point theory.
\newblock {\em Nonlinear Analysis: Theory, Methods \& Applications}, 24(3):385
  -- 407, 1995.

\bibitem{RASS}
R.~Benedetti and J.~Risler.
\newblock {\em Real Algebraic and Semi-algebraic Sets}.
\newblock Hermann, Paris, 1991.

\bibitem{Bochnak1998}
J.~Bochnak, M.~Coste, and M.-F. Roy.
\newblock {\em Real Algebraic Geometry}, volume~36.
\newblock Springer-Verlag, New York, 1998.

\bibitem{BDL}
J.~Bolte, A.~Daniilidis, and A.~S. Lewis.
\newblock Generic optimality conditions for semialgebraic convex programs.
\newblock {\em Mathematics of Operations Research}, 36(1):55--70, 2011.

\bibitem{BNTPS}
J.~Bolte, T.~P. Nguyen, J.~Peypouquet, and B.~W. Suter.
\newblock From error bounds to the complexity of first-order descent methods
  for convex functions.
\newblock {\em Mathematical Programming}, 165(2):471--507, 2017.

\bibitem{Bounkhel2012}
M.~Bounkhel.
\newblock {\em Regularity Concepts in Nonsmooth Analysis, Theory and
  Applications}.
\newblock Springer Optimization and Its Applications book series ({SOIA}), Vol.
  59. Springer, New York, 2012.

\bibitem{CHANEY1990649}
R.~W. Chaney.
\newblock Piecewise ${C}^k$ functions in nonsmooth analysis.
\newblock {\em Nonlinear Analysis: Theory, Methods \& Applications}, 15(7):649
  -- 660, 1990.

\bibitem{Clarke1990}
F.~Clarke.
\newblock {\em Optimization and Nonsmooth Analysis}.
\newblock Society for Industrial and Applied Mathematics, 1990.

\bibitem{Cox-Little-OShea:UAG2005}
D.~A. Cox, J.~Little, and D.~O'Shea.
\newblock {\em Using Algebraic Geometry}.
\newblock Graduate Texts in Mathematics. Springer-Verlag, Berlin-Heidelberg-New
  York, 2005.

\bibitem{Cui2018-2}
Y.~Cui, T.~H. Chang, M.~Hong, and J.~S. Pang.
\newblock On the finite number of directional stationary values of piecewise
  programs.
\newblock 2018.
\newblock arXiv:1803.00190.

\bibitem{Cui2018-1}
Y.~Cui and J.~S. Pang.
\newblock A study of piecewise linear-quadratic programs.
\newblock 2018.
\newblock arXiv:1709.05758v2.

\bibitem{DHP2017}
S.~T. Dinh, H.~V. Ha, and T.~S. Pham.
\newblock H{\"o}lder-type global error bounds for non-degenerate polynomial
  systems.
\newblock {\em Acta Mathematica Vietnamica}, 42(3):563--585, 2017.

\bibitem{DIL}
D.~Drusvyatskiy, A.~D. Ioffe, and A.~S. Lewis.
\newblock Generic minimizing behavior in semialgebraic optimization.
\newblock {\em SIAM Journal on Optimization}, 26(1):513--534, 2016.

\bibitem{FHKO}
M.~J. Fabian, R.~Henrion, A.~Y. Kruger, and J.~Outrata.
\newblock Error bounds: Necessary and sufficient conditions.
\newblock {\em Set-Valued and Variational Analysis}, 18(2):121--149, 2010.

\bibitem{FO1982}
O.~Fujiwara.
\newblock Morse programs: A topological approach to smooth constrained
  optimization.
\newblock {\em Mathematics of Operations Research}, 7(4):602--616, 1982.

\bibitem{DRMD}
I.~M. Gelfand, M.~Kapranov, and A.~Zelevinsky.
\newblock {\em Discriminants, Resultants, and Multidimensional Determinants}.
\newblock Mathematics: Theory \& Applications. Birkh\"auser, 1994.

\bibitem{DT}
V.~Guillemin and A.~Pollack.
\newblock {\em Differential Topology}.
\newblock Prentiee-HaIl, New Jersey, 1974.

\bibitem{HaHV2017}
H.~V. H\`a and T.~S. Ph\d{a}m.
\newblock {\em Genericity in polynomial optimization}.
\newblock World Scientific Publishing, Singapore, 2017.

\bibitem{Hager1979}
W.~W. Hager.
\newblock Lipschitz continuity for constrained processes.
\newblock {\em SIAM Journal on Control and Optimization}, 17(3):321--338, 1979.

\bibitem{JP1988}
H.~Jongen and D.~Pallaschke.
\newblock On linearization and continuous selections of functions.
\newblock {\em Optimization}, 19(3):343--353, 1988.

\bibitem{Kruger2015}
A.~Y. Kruger.
\newblock Error bounds and {H}{\"o}lder metric subregularity.
\newblock {\em Set-Valued and Variational Analysis}, 23(4):705--736, 2015.

\bibitem{KUNTZ1995197}
L.~Kuntz and S.~Scholtes.
\newblock Qualitative aspects of the local approximation of a piecewise
  differentiable function.
\newblock {\em Nonlinear Analysis: Theory, Methods \& Applications}, 25(2):197
  -- 215, 1995.

\bibitem{LeePham2016}
G.~M. Lee and T.~S. Ph{\d a}m.
\newblock Stability and genericity for semi-algebraic compact programs.
\newblock {\em Journal of Optimization Theory and Applications},
  169(2):473--495, 2016.

\bibitem{LeePham2017}
G.~M. Lee and T.~S. Ph{\d a}m.
\newblock Generic properties for semialgebraic programs.
\newblock {\em SIAM Journal on Optimization}, 27(3):2061--2084, 2017.

\bibitem{Li2013}
G.~Li.
\newblock Global error bounds for piecewise convex polynomials.
\newblock {\em Mathematical Programming}, 137(1):37--64, 2013.

\bibitem{LMNP}
G.~Li, B.~S. Mordukhovich, T.~T.~A. Nghia, and T.~S. Ph{\d a}m.
\newblock Error bounds for parametric polynomial systems with applications to
  higher-order stability analysis and convergence rates.
\newblock {\em Mathematical Programming}, 168(1):313--346, 2018.

\bibitem{LMP2015}
G.~Li, B.~S. Mordukhovich, and T.~S. Ph{\d a}m.
\newblock New fractional error bounds for polynomial systems with applications
  to {H}{\"o}lderian stability in optimization and spectral theory of tensors.
\newblock {\em Mathematical Programming}, 153(2):333--362, 2015.

\bibitem{Milnor1968}
J.~Milnor.
\newblock {\em Singular Points of Complex Hypersurfaces}, volume~61 of {\em
  Annals of Mathematics Studies}.
\newblock Princeton University Press, Princeton, 1968.

\bibitem{Mordukhovich2006}
B.~S. Mordukhovich.
\newblock {\em Variational Analysis and Generalized differentiation, I: Basic
  Theory, II: Applications}.
\newblock Springer, Berlin, 2006.

\bibitem{NZ2001}
K.~F. Ng and X.~Y. Zheng.
\newblock Error bounds for lower semicontinuous functions in normed spaces.
\newblock {\em SIAM Journal on Optimization}, 12(1):1--17, 2001.

\bibitem{NieDisNon}
J.~Nie.
\newblock Discriminants and nonnegative polynomials.
\newblock {\em Journal of Symbolic Computation}, 47(2):167--191, 2012.

\bibitem{NieFiniteCon}
J.~Nie.
\newblock Optimality conditions and finite convergence of {Lasserre}'s
  hierarchy.
\newblock {\em Mathematical Programming, Ser. A}, 146(1--2):97--121, 2014.

\bibitem{Pang1997}
J.-S. Pang.
\newblock Error bounds in mathematical programming.
\newblock {\em Mathematical Programming}, 79(1):299--332, 1997.

\bibitem{PD1996}
J.-S. Pang and D.~Ralph.
\newblock Piecewise smoothness, local invertibility, and parametric analysis of
  normal maps.
\newblock {\em Mathematics of Operations Research}, 21(2):401--426, 1996.

\bibitem{PT2001}
G.~Pataki and L.~Tun{\c c}el.
\newblock On the generic properties of convex optimization problems in conic
  form.
\newblock {\em Mathematical Programming}, 89(3):449--457, 2001.

\bibitem{DS1997}
D.~Ralph and S.~Scholtes.
\newblock Sensitivity analysis of composite piecewise smooth equations.
\newblock {\em Mathematical Programming}, 76(3):593--612, 1997.

\bibitem{R2003}
R.~T. Rockafellar.
\newblock A property of piecewise smooth functions.
\newblock {\em Computational Optimization and Applications}, 25(1):247--250,
  2003.

\bibitem{Rockafellar98}
R.~T. Rockafellar and R.~Wets.
\newblock {\em Variational Analysis}.
\newblock Grundlehren der Mathematischen Wissenschaften, Vol. 317. Springer,
  New York, 1998.

\bibitem{SS1973}
R.~Saigal and C.~Simon.
\newblock Generic properties of the complementarity problem.
\newblock {\em Mathematical Programming}, 4(1):324--335, 1973.

\bibitem{ScholtesBook}
S.~Scholtes.
\newblock {\em Introduction to Piecewise Differentiable Equations}.
\newblock Springer-Verlag, New York, 2012.

\bibitem{SA1997}
A.~Shapiro.
\newblock First and second order analysis of nonlinear semidefinite programs.
\newblock {\em Mathematical Programming}, 77(1):301--320, 1997.

\bibitem{SR1979}
J.~E. Spingarn and R.~T. Rockafellar.
\newblock The generic nature of optimality conditions in nonlinear programming.
\newblock {\em Mathematics of Operations Research}, 4(4):425--430, 1979.

\bibitem{Dries1996}
L.~van~den Dries and C.~Miller.
\newblock Geometric categories and o-minimal structures.
\newblock {\em Duke Mathematical Journal}, 84:497--540, 1996.

\bibitem{Womersley1982}
R.~S. Womersley.
\newblock Optimality conditions for piecewise smooth functions.
\newblock In D.~C. Sorensen and R.~J.-B. Wets, editors, {\em Nondifferential
  and Variational Techniques in Optimization}, pages 13--27. Springer Berlin
  Heidelberg, Berlin, Heidelberg, 1982.

\end{thebibliography}

\end{document}